\documentclass[12pt]{amsart}
\usepackage{amsmath,amsfonts,amsthm,amscd,amssymb,mathrsfs,amssymb, multirow}
\usepackage{tikz-cd}
\usepackage{stmaryrd}
\usepackage{graphicx}
\usepackage{verbatim}
\usepackage{color}
\newtheorem{theorem}{Theorem}

\newtheorem{proposition}[theorem]{Proposition}
\newtheorem{lemma}[theorem]{Lemma}
\newtheorem{definition}[theorem]{Definition}

\numberwithin{figure}{section}
\newtheorem{remark}[theorem]{Remark}

\renewcommand{\H}{\mathcal{H}}
\newcommand{\Z}{\mathcal{Z}}
\newcommand{\CP}{\mathbb{CP}}

\newcommand{\CC}{\mathbb{C}}

\newcommand{\RR}{\mathbb{R}}
\newcommand{\ZZ}{\mathbb{Z}}

\newcommand{\WW}{\mathbb{W}}
\newcommand{\VV}{\mathbb{V}}
\newcommand{\U}{{\rm{U}}}

\newcommand{\oo}{\mathcal{O}}

\numberwithin{equation}{section}
\numberwithin{theorem}{section}
\numberwithin{table}{section}
\numberwithin{table}{section}
\begin{document}
\bibliographystyle{amsalpha} 
\title[Deformation theory of scalar-flat K\"ahler ALE surfaces]{Deformation theory of scalar-flat \\ K\"ahler ALE surfaces}
\author{Jiyuan Han}
\address{Department of Mathematics, University of Wisconsin, Madison, 
WI, 53706}
\email{jiyuan@math.wisc.edu}
\author{Jeff A. Viaclovsky}
\address{Department of Mathematics, University of California, Irvine, 
CA, 92697}
\email{jviaclov@uci.edu}
\thanks{The authors were partially supported by NSF Grant DMS-1405725. This paper was completed while the authors were in residence at Mathematical Sciences Research Institute in Berkeley, California. The authors would like to thank MSRI for their support, and for providing such an excellent working environment}
%\date{November 7, 2016, revised version March 20, 2018.}
\begin{abstract} 
We prove a Kuranishi-type theorem for deformations of complex structures on ALE K\"ahler surfaces. This is used to prove that for any scalar-flat K\"ahler ALE surface, all small deformations of complex structure also admit scalar-flat K\"ahler ALE metrics. A local moduli space of scalar-flat K\"ahler ALE metrics is then constructed, which is shown to be universal up to small diffeomorphisms (that is, diffeomorphisms which are close to the identity in a suitable sense). A formula for the dimension of the local moduli space is proved in the case of a scalar-flat K\"ahler ALE surface which deforms to a minimal resolution of 
$\CC^2/\Gamma$, where $\Gamma$ is a finite subgroup of ${\rm{U}}(2)$ without complex reflections.  
 \end{abstract}
\maketitle
\vspace{-5mm}
\setcounter{tocdepth}{1}
%\tableofcontents
\vspace{-2mm}
%%%%%%%%%%%%%%%%%%%%%%%%%%%%%%%%%%%%%%%%%%%%%%%%%%%%%%%%
\section{Introduction}
\label{intro}
%%%%%%%%%%%%%%%%%%%%%%%%%%%%%%%%%%%%%%%%%%%%%%%%%%%%%
This article is concerned with the following class of metrics:
\begin{definition}
An ALE K\"ahler surface $(X, g,J)$ is a K\"ahler manifold of complex dimension $2$ with the following property. There exists a compact subset $K\subset X$ and a diffeomorphism 
$\psi: X \setminus K\rightarrow (\RR^4\setminus \overline{B})/\Gamma$, such that for each multi-index $\mathcal{I}$ of order~$|\mathcal{I}|$
\begin{align}
\partial^{\mathcal{I}}(\psi_*(g)-g_{Euc}) = O(r^{-\mu-|\mathcal{I}|}),
\end{align}
as $r \rightarrow \infty$, where $\Gamma$ is a finite subgroup of ${\rm{U}}(2)$ containing no complex reflections, $B$ denotes a ball centered at the origin, and $g_{Euc}$ denotes the Euclidean metric. The real number $\mu$ is called the order of $g$.
\end{definition}
\begin{remark}{\em
In this paper, henceforth $\Gamma$ will always be a finite subgroup of ${\rm{U}}(2)$ containing no complex reflections.
}
\end{remark}
We are interested in the class of scalar-flat K\"ahler ALE metrics. These 
are interesting since they are {\it{extremal}} in the sense of Calabi \cite{Calabi85},
and they arise as ``bubbles'' in gluing constructions for extremal K\"ahler metrics
\cite{ALM_kummer, ALM_resolution, Arezzo06, APSinger, BiquardRollin, RollinSinger, RS09}. In the case of scalar-flat K\"ahler ALE metrics, it is known that there exists an ALE coordinate system for which the order of such a metric is at least $2$ \cite{LeBrunMaskit}. 

We note that for an ALE K\"ahler metric of order $\mu$, there exist ALE coordinates for which 
\begin{align}
\partial^{\mathcal{I}} (J - J_{Euc})= O(r^{-\mu-|\mathcal{I}|}),
\end{align}
for any multi-index $\mathcal{I}$ as $r \rightarrow \infty$, where $J_{Euc}$ is the standard complex structure on Euclidean space \cite{HL15}. 

There are many known examples of scalar-flat K\"ahler ALE metrics. 
In the case that $\Gamma \subset {\rm{SU}}(2)$, Kronheimer constructed and classified the hyperk\"ahler ALE metrics \cite{Kr89, KrII89}. Calderbank-Singer constructed a family of scalar-flat K\"ahler ALE metrics on the minimal resolution of any cyclic quotient singularity \cite{CalderbankSinger}. For the remaining subgroups of ${\rm{U}}(2)$, the existence of scalar-flat K\"ahler metrics on the minimal resolution of $\CC^2/\Gamma$, was shown by Lock-Viaclovsky \cite{LV14}.

\iffalse
\begin{itemize}

\item ${\rm{SU}}(2)$ case: when $\Gamma\subset {\rm{SU}}(2)$, Kronheimer has constructed families of hyperk\"ahler ALE metrics \cite{Kr89} on manifolds diffeomorphic to the 
minimal resolution of $\CC^2 / \Gamma$. In \cite{KrII89}, Kronheimer also proved a Torelli-type theorem classifying hyperk\"ahler ALE surfaces. In the $A_k$ case, these metrics were previously discovered by Eguchi-Hanson for $k=1$ \cite{EguchiHanson}, and 
by Gibbons-Hawking for all $k \geq 1$ \cite{GibbonsHawking}. 

\vspace{2mm}
\item Cyclic case: For the $\frac{1}{p} (1,q)$-action, Calderbank-Singer constructed a family of scalar-flat K\"ahler ALE metrics on the minimal resolution of any cyclic quotient singularity \cite{CalderbankSinger}. These metrics are toric and come in families of dimension $k-1$, where $k$ is the length of the corresponding Hirzebruch-Jung algorithm.  For $q=1$ and $q=p-1$, these metrics are the LeBrun negative mass metrics
and the toric multi-Eguchi-Hanson metrics, respectively \cite{LeBrunnegative, GibbonsHawking}.

\vspace{2mm}
\item Non-cyclic non-${\rm{SU}}(2)$ case: The existence of scalar-flat K\"ahler metrics on the minimal resolution of $\CC^2/\Gamma$, was shown by Lock-Viaclovsky \cite{LV14}.

\end{itemize}
\fi

The question we address in this paper is whether the scalar-flat K\"ahler property is preserved under small deformations of complex structure. 
In the cases where $\Gamma \subset {\rm{SU}}(2)$, 
the hyperk\"ahler quotient construction
produces hyperk\"ahler metrics for the minimal
resolution complex structure as well as for all small deformations of the minimal resolution complex structure.
In the case of the LeBrun negative mass metrics, using arguments from twistor theory, 
Honda has shown that all small 
deformations of the complex structure on $\mathcal{O}(-n)$ also admit 
scalar-flat K\"ahler metrics \cite{HondaOn, Honda_2014}. 
For the case of general $\Gamma$, in \cite{LV14}, employing Honda's result, it was
shown that {\textit{some}} of small deformations of complex structure 
admit scalar-flat K\"ahler ALE metrics. 

Our main result in this paper shows that for {\textit{any}} scalar-flat K\"ahler ALE surface, 
{\textit{all}} small deformations of complex structure admit scalar-flat 
K\"ahler ALE metrics. Our proof is analytic in nature, and in particular, gives a new analytic proof of Honda's result on $\mathcal{O}(-n)$ mentioned above.

To state our result precisely, we 
next recall some basic facts regarding deformations of complex structures. 
For a complex manifold $(X,J)$, let $\Lambda^{p,q}$ denote the bundle of $(p,q)$-forms, and let $\Theta$ denote the holomorphic tangent bundle. The deformation complex 
corresponds to a real complex as shown in the commutative diagram 
\begin{equation}
\begin{tikzcd}[column sep=2.1cm]
\label{cd1}
\Gamma(\Theta) \arrow[r, "\overline{\partial}"] \arrow[d, "Re"] & \Gamma( \Lambda^{0,1} \otimes \Theta) \arrow[r, "{\overline{\partial}}"] \arrow[d, "Re"] & \Gamma(\Lambda^{0,2} \otimes \Theta) \arrow[d, "Re"]\\
\Gamma(TX) \arrow[r,"Z \mapsto -\frac{1}{2} J \circ \mathcal{L}_{Z}J"]  & \Gamma(End_{a}(TX)) 
\arrow[r, "I \mapsto \frac{1}{4} J \circ N_J'(I)"] &  \Gamma\big( \{\Lambda^{0,2} \otimes \Theta \oplus \Lambda^{2,0} \otimes \overline{\Theta} \}_{\RR} \big),
\end{tikzcd}
\end{equation}
where $\mathfrak{L}_Z J$ is the Lie derivative of $J$,
\begin{align}
End_{a}(TX) = \{I\in End(TX): IJ = -JI\},
\end{align}
and $N_J'$ is the linearization of Nijenhuis tensor 
\begin{align}
N(X,Y) = 2\{[JX,JY]-[X,Y]-J[X,JY]-J[JX,Y]\}
\end{align}
at $J$. 
Each isomorphism $Re$ is simply taking the real part of a section. 
If $g$ is a Hermitian metric compatible with $J$, then let $\square$ denote the $\bar\partial$-Laplacian 
\begin{align}
\square  \equiv \bar\partial^*\bar\partial+\bar\partial\bar\partial^*,
\end{align}
where $\bar\partial^*$ denotes the formal $L^2$-adjoint. 
Each complex bundle in the diagram \eqref{cd1} admits a $\square$-Laplacian,
and these correspond to real Laplacians on each real bundle 
in \eqref{cd1}. We will use the same $\square$-notation for 
these real Laplacians. 

To state our most general result, we need the following definition.
This is necessary because there is a gauge freedom of Euclidean 
motions in the definition of ALE coordinates.   
\begin{definition}
{\em
Let $(X,g,J)$ be a K\"ahler ALE surface. For any bundle $E$ in 
the diagram \eqref{cd1}, and $\tau \in \RR$, define
\begin{align}
\H_{\tau}(X,E) &= \{\theta\in \Gamma(X,E): \square\theta=0, \theta=O(r^{\tau}) 
 \mbox{ as } r \rightarrow \infty \}.
\end{align}
Define 
\begin{align}
\WW = \{ Z \in \mathcal{H}_1(X, TX) \ | \ \mathfrak{L}_Z g  = O (r^{-1}), \ \mathfrak{L}_Z J = O(r^{-3}),\mbox{ as } r \rightarrow \infty \}.
\end{align}
Finally, define the real subspace
\begin{align}
\H_{ess} (X,End_a(TX)) \subset  \H_{-3}(X,End_a(X)) 
\end{align}
({\it{ess}} is short for essential) to be the $L^2$-orthogonal complement in $\H_{-3}(X,End_a(X))$
of the subspace 
\begin{align}
\VV = \{ \theta \in \H_{-3}(X,End_a(TX)) \  | \ 
\theta =  J\circ \mathfrak{L}_Z J, \  Z \in \WW \}.
\end{align}
}
\end{definition}
Our main result is the following.
\begin{theorem}
\label{tmain}
Let $(X,g,J)$ be a scalar-flat K\"ahler ALE surface.
Let $-2 < \delta < -1$, $0 < \alpha < 1$,  and $k$ an integer with $k \geq 4$
 be fixed constants. 
Let $B^1_{\epsilon_1}$ denote an $\epsilon_1$-ball in $\H_{ess} ( X,End_a(TX)) $, 
$B^2_{\epsilon_2}$ denote an $\epsilon_2$-ball in $\H_{-3}(X,\Lambda^{1,1})$ (both using the $L^2$-norm). 
Then there exists $\epsilon_1 >0$ and $\epsilon_2 >0$ and a family $\mathfrak{F}$ of scalar-flat K\"ahler metrics near $g$, parametrized by 
$B^1_{\epsilon_1} \times B^2_{\epsilon_2}$, that is, there is a differentiable mapping 
\begin{align}
F : B^1_{\epsilon_1} \times B^2_{\epsilon_2} \rightarrow {Met}(X), 
\end{align}
into the space of smooth Riemannian metrics on $X$, 
with $\mathfrak{F} = F(B^1_{\epsilon_1} \times B^2_{\epsilon_2})$ satisfying the following ``versal'' property:
there exists a constant $\epsilon_3 >0$ 
such that for any scalar-flat K\"ahler metric $\tilde{g}\in B_{\epsilon_3}(g)$,
there exists a diffeomorphism $\Phi: X \rightarrow X$, $\Phi \in C^{k+1,\alpha}_{loc}$,
such that $\Phi^* \tilde{g} \in \mathfrak{F}$, 
where 
\begin{align}
\label{wnd}
B_{\epsilon_3}(g) =  \{g' \in C^{k,\alpha}_{loc}(S^2(T^*X))  \ | \ \Vert g - g' \Vert_{C^{k,\alpha}_{\delta}(S^2(T^*X))} < \epsilon_3 \}.
\end{align} 
\end{theorem}
\begin{remark}
{\em
The norm in \eqref{wnd} is a certain weighted H\"older norm, 
see Section \ref{Notationsec} for the precise definition. For a more precise 
description of the diffeomorphism $\Phi$, see Theorem \ref{gthm} below. 
}
\end{remark}
In order to state our next result, we must recall some facts about 
vector fields and diffeomorphisms. 
If $(X,g)$ is an ALE metric, and $Y$ is a vector field on $X$, 
the Riemannian exponential mapping $\exp_p : T_pX 
\rightarrow X$ induces a mapping 
\begin{align}
\Phi_Y : X \rightarrow X
\end{align}
by 
\begin{align}
\label{phidef}
\Phi_Y (p) = \exp_p (Y).
\end{align}
If $Y \in C^{k,\alpha}_{s}(TX)$ has sufficiently small norm, 
($s < 0$ and $k$ will be determined in specific cases)
then $\Phi_Y$ is a diffeomorphism. We will use the correspondence 
$Y \mapsto \Phi_Y$ to parametrize a neighborhood of the identity, 
analogous to  \cite{Biquard06}.
\begin{definition}
{\em 
We say that $\Phi: X \rightarrow X$  
is a {\it{small diffeomorphism}} if $\Phi$ is of the form 
$\Phi = \Phi_Y$ for some vector field $Y$ satisfying
\begin{align}
\Vert Y \Vert_{C^{k+1,\alpha}_{\delta+1}} < \epsilon_4
\end{align}
for some $\epsilon_4 >0$ sufficiently small which depends on $\epsilon_3$.
}
\end{definition}
The family $\mathfrak{F}$ is not necessarily ``universal'',
because some elements in $\mathfrak{F}$ might be isometric.
However, the following theorem shows that after taking a quotient 
by an action of the holomorphic isometries of the central fiber $(X,g,J)$, the family $\mathfrak{F}$ is in fact universal (up to small diffeomorphisms).
\begin{theorem}
\label{t2}
Let $(X,g,J)$ be as in Theorem \ref{tmain},
and let $\mathfrak{G}$ denote the group of holomorphic isometries of $(X,g,J)$. 
Then there is an action of $\mathfrak{G}$ on $\mathfrak{F}$ with the following properties.
\begin{itemize}

\item Two metrics in $\mathfrak{F}$ are isometric 
if they are in the same orbit of $\mathfrak{G}$. 

\item If two metrics in $\mathfrak{F}$ are isometric by a small
diffeomorphism then they must be the same.  
\end{itemize}
\end{theorem}
Since each orbit represents a unique isometry class
of metric (up to small diffeomorphism), we will refer to the quotient $\mathfrak{M} = \mathfrak{F}/\mathfrak{G}$ as the ``local moduli space
of scalar-flat K\"ahler ALE metrics near $g$.''
The local moduli space $\mathfrak{M}$ is not a manifold in general, but
its dimension is in fact well-defined, and we define 
\begin{align}
\label{mdefi}
m = \dim ( \mathfrak{M}).
\end{align}

\begin{remark}{\em
We should point out that our local moduli space of metrics contains small rescalings, i.e, $g\mapsto \frac{1}{c^2}g(c\cdot,c\cdot)$ for $c$ close to $1$. If one considers scaled 
metrics as equivalent (which we do not), then the dimension would decrease by $1$.
}
\end{remark}

\subsection{Deformations of the minimal resolution}

As mentioned above, there are families of examples of scalar-flat K\"ahler ALE metrics on minimal resolutions of isolated quotient singularities. For convenience, we next recall
the definition of a minimal resolution.
\begin{definition}
\label{minresdef}
{\em Let $\Gamma \subset \U(2)$ be as above.  Then, a smooth complex surface ${X}$ is called a {\textit{minimal resolution}} 
of $\CC^2/\Gamma$ if there is a mapping $\pi : {X}\rightarrow \CC^2/\Gamma$ such 
that 
the restriction $\pi:{X} \setminus \pi^{-1}(0)\rightarrow \CC^2/\Gamma\setminus\{0\}$ is a biholomorphism, and
$\pi^{-1}(0)$ is a divisor in ${X}$ containing no $-1$ curves.
The divisor $\pi^{-1}(0)$ is called the {\it{exceptional divisor}}. }
\end{definition}

In the cyclic case, the exceptional divisor is a string of rational curves with normal crossing singularities. In the case that $\Gamma$ is non-cyclic, it was 
shown by Brieskorn, see \cite{Brieskorn}, that the exceptional divisor is a tree of 
rational curves with normal crossing singularities. 
There are three Hirzebruch-Jung strings attached to a single curve,
called the {\textit{central rational curve}}. 
The self-intersection number of this 
curve will be denoted $-b_{\Gamma}$, and the total number 
of rational curves will be denoted by~$k_{\Gamma}$.

A specialized version of Theorem \ref{tmain} is the following.
\begin{theorem}
\label{generalthm}
Let $(X,g,J)$ be any scalar-flat K\"ahler ALE metric on the minimal resolution of $\CC^2/\Gamma$, where $\Gamma \subset {\rm{U}}(2)$ is as above. Define 
\begin{align}
j_{\Gamma} =  2 \sum_{i = 1}^{k_{\Gamma}} (e_i -1),
\end{align}
where $-e_i$ is the self-intersection number of the $i$th rational curve, and $k_{\Gamma}$ is the number of rational curves in the exceptional divisor, and let 
\begin{align}
d_{\Gamma}=  j_{\Gamma} + k_{\Gamma}.
\end{align}
Then there is a family, $\mathfrak{F}$, parametrized by a ball in $\RR^{d_{\Gamma}}$, 
of scalar-flat K\"ahler metrics near $g$ with the ``versal'' property stated in Theorem \ref{tmain}.
\end{theorem}
We let $m_{\Gamma}$ denote the dimension of the moduli space near the minimal resolution of $\CC^2/\Gamma$, where the dimension is defined in \eqref{mdefi}. Note $m_{\Gamma} < d_{\Gamma}$ due to the action of the automorphism group, and $m_{\Gamma}$ can be explicitly computed for all groups $\Gamma$, but those computations will not be done here. 
\iffalse
\begin{table}[h]
\caption{Dimension of local moduli space of scalar-flat K\"ahler metrics}
\label{dimtable}
\begin{tabular}{ll l l}
\hline
$\Gamma\subset{\rm U}(2)$ & $d_{\Gamma}$ & $m_{\Gamma}$\\
\hline\hline
\vspace{1mm}
$ \frac{1}{3}(1,1)$ & $5$ & $2$\\
\vspace{1mm}
$ \frac{1}{p}(1,1) , p \geq 4$ &  $ 2p -1   $ & $2 p - 5$\\
\vspace{1mm}
$ \frac{1}{p}(1,q), q \neq 1, p-1 $ & $j_{\Gamma} + k_{\Gamma}$ & $j_{\Gamma} + k_{\Gamma} -2$ \\
non-cyclic, not in ${\rm{SU}}(2)$ & $j_{\Gamma} + k_{\Gamma}$  & $ j_{\Gamma} + k_{\Gamma} -1$\\
\hline
\end{tabular}
\end{table}
\fi
\begin{remark}
\label{hrem}
{\em
The dimension of the moduli space of hyperk\"ahler metrics 
is known to be $3 k - 3$ in the $A_k, D_k$ and $E_k$ cases for $k \geq 2$, 
and equal to $1$ in the $A_1$ case \cite{KrCR}. Our method of
parametrizing by complex structures and K\"ahler classes overcounts 
in this case (i.e., $F$ is not injective), since a hyperk\"ahler metric is K\"ahler with respect to a
$2$-sphere of complex structures.
}
\end{remark}
Our final result applies to a generic deformation of the minimal resolution. 
\begin{theorem}
\label{t1.8}
Let $(X,g,J)$ be any scalar-flat K\"ahler ALE surface which deforms to 
the minimal resolution of $\CC^2/\Gamma$ through a path $(X,g_t,J_t)$ $(0\leq t\leq 1)$, 
where $g_1 = g$, $g_0$ is the minimal resolution, 
and $\|g_t-g_s\|_{C^{k,\alpha}_\delta(g_0)}\leq C\cdot |s-t|$ with $C>0$ a uniform constant for
any $0\leq s,t\leq 1$, $k\geq 4$, $-2<\delta<-1$.
If $\mathfrak{G}(g)  = \{e\}$ then the local moduli space 
$\mathfrak{F}$ is smooth near $g$ and is a manifold of dimension $m = m_{\Gamma}$.
\end{theorem}

\iffalse
\begin{remark}{\em
The assumption that $X$ deforms to the minimal resolution of $\CC^2/\Gamma$ is reasonable. In \cite[Section 8.9]{Joyce2000}, Joyce remarks that any K\"ahler ALE surface with group $\Gamma\subset {\rm{U}}(2)$ is birational to a deformation of $\CC^2/\Gamma$ (a complete proof of this fact can be found in \cite{HRS16}). There are several possible components of the deformation of such a cone, here we consider surfaces in the ``Artin component'' of deformations of $\CC^2/\Gamma$.  We note however that there are some 
known examples of scalar-flat K\"ahler metrics on non-Artin components, which are free quotients of 
hyperk\"ahler metrics of $A_k$-type, see for example \cite{Suvaina_ALE}.  
}
\end{remark}
\fi

\begin{remark}{\em
It was recently shown that K\"ahler ALE surface with group $\Gamma\subset {\rm{U}}(2)$ is birational to a deformation of $\CC^2/\Gamma$ \cite{HRS16}. There are several possible components of the deformation of such a cone, here we consider surfaces in the ``Artin component'' of deformations of $\CC^2/\Gamma$.  We note however that there are some 
known examples of scalar-flat K\"ahler metrics on non-Artin components, which are free quotients of 
hyperk\"ahler metrics of $A_k$-type, see for example \cite{Suvaina_ALE}.

}
\end{remark}
We end the introduction with an outline of the paper. In Section \ref{Notationsec}, we begin with the definitions of the weighted H\"older spaces which will be used throughout the paper. Then, in Section \ref{SDCS} we give some analysis of the complex analytic compactifications, due to Hein-LeBrun-Maskit \cite{HL15, LeBrunMaskit}, of K\"ahler ALE spaces. In Section~\ref{Kuranishi}, we study the deformation of complex structures using an adaptation of Kuranishi's theory to ALE spaces. The main point is that since the manifold is non-compact, the sheaf cohomology group $H^1(X, \Theta)$ should be replaced by an appropriate space of decaying harmonic forms. In Section~\ref{vkf}, several key results about
gauging and diffeomorphisms are proved, which are used to prove ``versality'' of the 
family constructed in Section~\ref{Kuranishi}. In Section~\ref{esssec}, a refined gauging procedure is carried out, to construct the Kuranishi family of ``essential'' deformations. In Section \ref{stable}, we generalize Kodaira-Spencer's stability theorem for K\"ahler structures to the ALE setting, using some arguments of Biquard-Rollin \cite{BiquardRollin}. In Section \ref{DSFK}, we adapt the LeBrun-Singer-Simanca theory of deformations of extremal K\"ahler metrics to the ALE setting \cite{LS93, LeBrun_Simanca}. In Section~\ref{vp}, we prove the versal property of the family $\mathfrak{F}$, using a local slicing theorem, and prove Theorem \ref{t2}. 
In Section~\ref{defminsec}, we restrict attention to the minimal resolution, and prove Theorems \ref{generalthm} and \ref{t1.8}.

\subsection{Acknowledgements}  The authors would like to thank Olivier Biquard, Ronan Conlon, Akira Fujiki, Ryushi Goto, Hans-Joachim Hein, Nobuhiro Honda, and Claude LeBrun, for numerous helpful discussions on deformations of complex structures. Olivier Biquard, Joel Fine, and Jason Lotay provided crucial assistance with the slicing arguments. We would also like to thank Michael Lock for valuable discussions on properties of subgroups of ${\rm{U}}(2)$. Finally, the authors would like to thank to the anonymous referee for detailed helpful remarks, and for suggesting numerous improvements to the paper.

\section{Notation}
\label{Notationsec}
We begin with the definition of weighted H\"older spaces, which will be used
throughout the paper.
\begin{definition}{\em
Let $E$ be a tensor bundle on $X$, with Hermitian metric $\Vert \cdot\Vert_h$. Let $\varphi$ be a smooth section of $E$. We fix a point $p_0\in X$, and define $r(p)$ to be the distance between $p_0$ and $p$. Then define 
\begin{align}
\Vert \varphi\Vert_{C^{0}_\delta} &:= \sup_{p\in X}\Big\{\Vert\varphi(p)\Vert_h\cdot (1+r(p))^{-\delta}\Big\}\\
\Vert \varphi\Vert_{C^{k}_\delta} &:= \sum_{|\mathcal{I}|\leq k}\sup_{p\in X}
\Big\{\Vert\nabla^{\mathcal{I}} \varphi(p)\Vert_h\cdot (1+r(p))^{-\delta+|\mathcal{I}|}\Big\},
\end{align}
where $\mathcal{I} = (i_1,\ldots,i_n),|\mathcal{I}|=\sum_{j=1}^n i_j$.
Next, define
\begin{align}
[\varphi]_{C^{\alpha}_{\delta-\alpha}} &:= \sup_{0<d(x,y)<\rho_{inj}}\Big\{\min\{r(x),r(y)\}^{-\delta+\alpha}\frac{\Vert\varphi(x)-\varphi(y)\Vert_h}{d(x,y)^\alpha}\Big\},
\end{align}
where $0<\alpha<1$, $\rho_{inj}$ is the injectivity radius, and $d(x,y)$ is the distance between $x$ and $y$. The meaning of the tensor norm is to use parallel transport along the unique minimal geodesic from  $y$ to $x$, and then take the norm of the difference 
at $x$. 
The weighted H\"older norm is defined by 
\begin{align}
\Vert\varphi\Vert_{C^{k,\alpha}_\delta} &:= \Vert \varphi\Vert_{C^{k}_\delta}+\sum_{|\mathcal{I}|=k}[\nabla^{\mathcal{I}} \varphi]_{C^{\alpha}_{\delta-k-\alpha}},
\end{align}
and the space $C^{k,\alpha}_{\delta}(X, E)$ is the closure of $\{\varphi \in C^{\infty}(X,E): \Vert \varphi\Vert_{C^{k,\alpha}_\delta}<\infty\}$.
}
\end{definition}
\begin{remark}{\em
The dual space of H\"older space is not a H\"older space, 
but for the purpose of computing the dimension of cokernel of a Fredholm operator
of order $o$,
\begin{align}
H : C^{k,\alpha}_{\delta} \rightarrow  C^{k - o,\alpha}_{\delta-o},
\end{align}
we consider the adjoint operator as mapping between H\"older spaces  
\begin{align}
H^* : C^{k,\alpha}_{-4- \delta +o} \rightarrow  C^{k - o,\alpha}_{-4-\delta},
\end{align}
since the adjoint weight to weight $\delta$ is $- 4-\delta$, and using elliptic regularity. 
}
\end{remark}

\section{ALE K\"ahler surfaces}
\label{SDCS}
In this section, we will prove several results about ALE K\"ahler surfaces which 
will be needed later.  We note that the results in this section do not use the scalar-flat assumption. 

In \cite{LeBrunMaskit, HL15}, LeBrun-Maskit and Hein-LeBrun analyzed the asymptotic behavior of the metric and complex structure near infinity of ALE K\"ahler surfaces. The next proposition gives a summary of their results. (See also \cite{ChiLi} for other 
related results on complex analytic compactifications). 
\begin{proposition}[Hein-LeBrun-Maskit]
\label{p1.4}
Let $X_{\infty}$ be an end of an ALE K\"ahler surface $(X,g,J)$ with metric $g$ asymptotic to Euclidean metric at rate 
\begin{align}
|\nabla_{Euc}^\mathcal{I}(g_{j,k}-\delta_{j,k}))| = O(\rho^{-|\mathcal{I}|-1-\epsilon})
\end{align} 
for $|\mathcal{I}|=0,1$ if $\epsilon>1/2$ or for $|\mathcal{I}|=0,\ldots,4$ if $\epsilon\in (0,1/2]$; here $\nabla_{Euc}$ denotes the Euclidean derivative. Then there is a surface $S$ containing an embedded holomorphic curve $C\simeq \CP^1$ with self-intersection 1, such that the universal cover $\tilde{X}_{\infty}$ of $X_{\infty}$ is biholomorphic to $S \setminus C$.

Let $(X,g,J)$ be a K\"ahler ALE surface, then $X$ has one end and can be analytically compactified to a smooth compact surface $\hat{X}$ by adding a tree of rational curves $E_{\infty}$. The surface $\hat{X}$ is a rational surface or a ruled surface. Also, $X$ can be compactified to an orbifold surface $\hat{X}_{orb}$ by adding a single rational curve at $\infty$.

Furthermore, if the order of $g$ satisfies $1<\mu<3$, then there exists an ALE coordinate, under which $|J-J_{Euc}|=O(r^{-3})$, $|g-g_{Euc}| = O(r^{-\mu})$. 

\end{proposition}
\begin{remark}\rm
The orbifold compactification $\hat{X}_{orb}$ has cyclic singularities on the rational curve at $\infty$. The smooth compactification $\hat{X}$ is obtained from $\hat{X}_{orb}$ by resolving these singularities.
\end{remark}
Consider the 
space of holomorphic vector fields on $X$ with at most growth 
of order~$\tau$
\begin{align}
\Z_\tau(X, \Theta) = \{ Z \in \Gamma(X, \Theta) \ | \ \bar\partial Z = 0, \ Z = O(r^\tau)  \text{ as } r \rightarrow \infty\}.
\end{align}
The first vanishing result we need is the following. 
\begin{proposition}
\label{p2.5}
Let $(X,g,J)$ be an ALE K\"ahler surface. If $\tau < 0$ then 
\begin{align}
\dim \Z_\tau(X, \Theta) = 0.
\end{align}
\end{proposition}
\begin{proof}
Let $X_\infty = X\setminus \overline{B(R)}$. Let $R$ be sufficiently large, then by Proposition \ref{p1.4}, the universal cover $\tilde{X}_\infty$ can be analytically compactified by adding a rational curve $C$ at infinity, where $C$ has self-intersection $+1$. By \cite{HL15}, there exists a smooth map $[\xi_1,\xi_2, f]$ that maps $S = \tilde{X}_\infty\cup C$ to $\CP^2$, which maps $C$ holomorphically to $\{f=0\}$ as a curve $\CP^1\subset \CP^2$, where 
\begin{align}
\xi_j = \{\text{holomorphic part}\}+O(|f|^3),
\end{align}
as $f \rightarrow 0$. Let 
\begin{align}
z_j = x_j+\sqrt{-1}y_j = \frac{\xi_j}{f}, (j=1,2)
\end{align}
where  $\{x_1,y_1,x_2,y_2\}$ gives an ALE coordinate on $\tilde{X}_\infty$, and $|J-J_{Euc}| = O(|z|^{-3})$ as $z \rightarrow \infty$. Let 
\begin{align}
\label{coordinate}
(v,w) = \Big( \frac{1}{z_1},\frac{z_2}{z_1} \Big)
\end{align}
where $\{v = 0\}$ represents the complement of one point in $C$.

First, let $\sigma$ be a decaying holomorphic vector field on $X$, which can be lifted to $\tilde{\sigma}$ on $\tilde{X}_\infty$. Note that
$\tilde{\sigma}$ can be extended to $\CC^2$ smoothly by using a cut-off function. 
Since $\bar\partial-\bar\partial_{Euc} = O(|z|^{-3})$ and $\bar\partial \tilde{\sigma}=0$ on $\tilde{X}_\infty$,
\begin{align} 
\bar\partial_{Euc}\tilde{\sigma} = \sum_{i,j} h_{i,j}d\bar{z}_i\otimes \frac{\partial}{\partial z_j}=O(|z|^{-3}|\sigma|),
\end{align}
as $z \rightarrow \infty$. 
By using the $\bar\partial_{Euc}$-Poincar\'e lemma, there exist 
\begin{align}
p_j = O(|z|^{-2}\sigma)\in C^{\infty}(\CC^2), \, \bar\partial_{Euc} p_j = \sum_i h_{i,j} d\bar{z}_i.
\end{align}
The formula for $p_j$ can be written out explicitly as follows. Let
\begin{align}
\label{Poin}
\begin{split}
q_j &= \frac{1}{2\pi\sqrt{-1}}\int_{\CC}h_{2,j}(z_1,\zeta_2)\frac{d\zeta_2\wedge d\bar{\zeta}_2}{\zeta_2-z_2} \\
q'_j &= \frac{1}{2\pi\sqrt{-1}}\int_{\CC}(h_{1,j}-\bar\partial_1 q_j)(\zeta_1,z_2)\frac{d\zeta_1\wedge d\bar{\zeta}_1}{\zeta_1-z_1},
\end{split}
\end{align}
where the integral formula is valid since $h_{i,j} = o(|z|^{-3})$. Then $p_j = q_j+q'_j$.
Then we have $\tilde{\sigma}-\sum_j p_j\frac{\partial}{\partial z_j}$ is a decaying $\bar\partial_{Euc}$-holomorphic vector field on $\CC^2$, which must vanish identically. 
The growth rate of $\sigma$ can be dropped by $2$ iteratively by this argument, so we can assume that $\sigma$ decays at any rate in the coordinates $\{z_1,z_2\}$.

On
$\tilde{X}_\infty$, there exists a holomorphic $(2,0)$-form 
$\Omega = dz_1\wedge dz_2+O(|z|^{-2+\epsilon})$ as 
$|z| \rightarrow \infty$. Let $\tilde{\sigma}^*\in \Gamma(\tilde{X}_\infty,\Omega^1)$ denote the contraction of $\tilde{\sigma}$ with $\Omega$. The section $\tilde{\sigma}^*$ is also holomorphic and we have $\tilde{\sigma}^* = O(|z|^{-10})$ as $z \rightarrow \infty$ since we can assume $\tilde{\sigma}$ decays to any order. The coordinate change gives: $dz_1 = \frac{-1}{v^2}dv, dz_2 = \frac{1}{v}dw-\frac{w}{v^2}dv$, $|z|^2 = \frac{1}{|v|^2}(1+|w|^2)$. Then $\tilde{\sigma}^* = O(|v|^8)$ near $C$ and can be extended to $C$ with $\tilde{\sigma}^* = 0$ on $C$. Consider the exact sequence
\begin{align}
0\rightarrow N_C^*\rightarrow \Omega^1|_C \rightarrow \Omega^1(C)\rightarrow 0,
\end{align}
where $\Omega^1|_C$ is the restriction of the bundle $\Omega^1$ on $C$, $\Omega^1(C)$ is the cotangent bundle of $C$ and $N_C$ is the normal bundle of $C$ in $S$. Then we have $\Omega^1|_C = \oo(-2)\oplus \oo(-1)$.
This implies $\tilde{\sigma}^*$ vanishes on any rational curve with self-intersection $+1$ on $S$. Since $C$ has self-intersection $+1$, $H^1(C,N_C^*\otimes \Theta(C)) = 0$, so $C$ is rigid in $S$. Then there is an open neighborhood $U\subset S$ of $C$, such that $\tilde{\sigma}^*$ vanishes identically over $U$. Then $\sigma$ vanishes on an open subset of $X$, which implies that $\sigma$ vanishes identically on $X$.

\end{proof}
Next, we consider harmonic $(0,2)$-forms with values in the holomorphic tangent bundle.
\begin{proposition}
\label{p2.6}
Let $X$ be a K\"ahler ALE surface. If $\delta < 0$, then $\H_\delta(X,\Lambda^{0,2}\otimes\Theta)=0$.
\end{proposition}
\begin{proof}
The proposition follows an argument in \cite[Theorem 4.2]{LeBrunMaskit},
with minor modifications. For completeness, we give a proof 
here following their idea.
Let $\sigma\in \H_\delta(X,\Lambda^{0,2}\otimes\Theta)$. Recall that the conjugated Hodge star operator $\bar{*}$ maps $\sigma$ to a $\square$-harmonic form $\sigma^*\in \H_\delta(X,\Lambda^{2,0}\otimes\Omega^1)$. Let $\tilde{\sigma}^*$ be its lifting on $\tilde{X}_\infty$ . Following the same notion as in the proof of Proposition \ref{p2.5}, we have 
\begin{align}
\tilde{\sigma}^* = \sum_j a_j d z_1\wedge d z_2\otimes d z_j + O(|z|^{-3}|\tilde{\sigma}^*|),
\end{align}
where $a_j$ are $\bar\partial_{Euc}$-holomorphic functions and $a_j = O(|z|^{\delta})$
as $z \rightarrow \infty$. As shown in the proof of Proposition \ref{p2.5}, $a_j = 0$, so $\tilde{\sigma}^* = O(|z|^{\delta-3})$ as $z \rightarrow \infty$. 
Since the growth rate of $\tilde{\sigma}^*$ can be dropped by $3$ iteratively, 
$\tilde{\sigma}^*$ can decay to any order under the coordinate of $\{z_1,z_2\}$, then $\tilde{\sigma}^*$ can be extended to $C$ with $\tilde{\sigma}^*=0$ on $C$. Since
\begin{align}
\Omega^{2,0}\otimes \Omega^1|_C = (\Omega^1(C)\oplus N_C^*)\otimes (\Omega^{2,0}(C)\otimes N_C^*),
\end{align}
we have
\begin{align}
\Omega^{2,0}\otimes\Omega^1|_C = (\oo(-2)\oplus\oo(-1))\otimes (\oo(-2)\otimes\oo(-1)) = \oo(-5)\oplus \oo(-4).
\end{align}
This implies that $\tilde{\sigma}^*$ vanishes on any rational curve with self-intersection $+1$ on $\tilde{X}_\infty$. By the same argument as in Proposition \ref{p2.5}, there is an open neighborhood $U\subset S$ of $C$, such that $\tilde{\sigma}^*$ vanishes identically over $U$. Then $\sigma$ vanishes on an open subset of $X$, which implies that $\sigma$ vanishes identically on $X$. 
\end{proof}

The next proposition shows that $\H_\delta(X,\Lambda^{0,1}\otimes\Theta)$ automatically has a faster decaying rate.
\begin{proposition}
\label{dcyprop}
Let $X$ be an ALE K\"ahler surface, and $\hat{X}$ be its analytic compactification. 
Then for any $-3 < \tau < 0$ we have 
\begin{align}
\label{aa1}
\H_{\tau}(X,\Lambda^{0,1}\otimes\Theta) =\H_{-3}(X,\Lambda^{0,1}\otimes\Theta),
\end{align}
and if $\Gamma$ is not cyclic or $\Gamma = \ZZ/2\ZZ$, then 
\begin{align}
\label{aa2}
\H_{\tau}(X,\Lambda^{0,1}\otimes\Theta) =\H_{-4}(X,\Lambda^{0,1}\otimes\Theta).
\end{align}
\end{proposition}
\begin{proof}
First, we will show that any decaying harmonic element $\phi\in H^0(X,\Lambda^{0,1}\otimes\Theta)$ has a decay rate of at least $O(r^{-3})$ at infinity. 
Since $X$ is K\"ahler, by \cite[Part 5]{Moroianu}, the operator $\square$ admits an expansion near infinity of the form 

\begin{align}
\square = \frac{1}{2}\nabla^*\nabla + \mathfrak{R} \
\end{align}
on $\Lambda^{0,1}\otimes\Theta$, where $\nabla$ is the covariant derivative on ${\Lambda^{0,1}\otimes\Theta}$, and $\mathfrak{R}$ is an operator given by curvature forms acting on the same bundle. The leading term of $\nabla^*\nabla$ is the Euclidean Laplacian $\Delta_{Euc}$, so we have
\begin{align}
\label{bochner}
\square = \frac{-1}{2}\Delta_{Euc}+Q,
\end{align}
where $Q = A(\nabla\phi)+Ric(\phi)$ denotes higher order terms. The element 
$\phi$ admits an expansion of the form 
\begin{align}
\phi = f + O(r^{-3 + \epsilon}),
\end{align}
as $r\to \infty$, where $f$ is of the form 
\begin{align}
f = \sum_{i,j}\frac{f_{i,j}}{r^2} d\bar{z}_i\otimes\frac{\partial}{\partial z_j},
\end{align}
and $f_{i,j}$ are constants. 
Since $\phi$ is $\square$-harmonic, both $\bar\partial f =0$ and $\bar\partial^* f=0$. It is easy to see that 
this implies that each $f_{i,j}=0$. Consequently, $\phi = O(r^{-3 + \epsilon})$
as $r \rightarrow \infty$. 
Since $\phi$ admits an expansion with harmonic leading term, we must have $\phi = O(r^{-3})$. 
 
Furthermore, if $\Gamma$ is not cyclic or $\Gamma = \ZZ/2\ZZ$, then $\Gamma$ contains the element $-1$. 
The leading term of $\phi$ is of the form 
\begin{align}
f = \sum_{k,l}\frac{A_{kl}}{r^4}d\bar{z}_k\otimes\frac{\partial}{\partial z_l},
\end{align}
where $A_{kl}$ is a linear combination of $z_1,\bar{z}_1, z_2, \bar{z}_2$. Since any such nonzero $f$ 
is not invariant under the action of $-1$, we must have $f=0$, and then $\phi = O(r^{-4})$. This finishes the proof of \eqref{aa1} and \eqref{aa2}.

\end{proof}
The following Proposition will be used in Section \ref{stable}. 
\begin{proposition}
\label{p2.8}
Let $(X,g,J)$ be an ALE K\"ahler surface. Then
\begin{align}
b_1(X) = 0,
\end{align}
where $b_1(X)$ denotes the first Betti number of $X$. 
Furthermore, for $\tau < 0$, 
\begin{align}
\label{aa4}
\dim (\H_{\tau}(X,\Lambda^{0,2})) = \dim (\H_{\tau}(X,\Lambda^{2,0})) = 0.
\end{align}
\end{proposition} 
\begin{proof}
Let $U$ be a tubular neighborhood of $E_\infty$, $\hat{X} = X\cup E_\infty$ be the analytic compactification of $X$. By Proposition \ref{p1.4}, $\hat{X}$ is either a rational surface or a ruled surface, then $b^1(\hat{X})=0$. Since $E_\infty$ is a tree of rational curves, $b^1(U)=0$ and $b^1(X\cap U)=0$. Then by  the Mayer-Vietoris theorem, $b^1(X)=0$.

Since $\dim(\H_\tau(X,\Lambda^{0,2}))=\dim(\H_\tau(X,\Lambda^{2,0}))$, we just need to show the latter one equals to zero. The proof is similar to the proof of Proposition \ref{p2.6}, so we will skip some details. 
For any $\sigma\in \H_\tau(X,\Lambda^{2,0})$, let $\tilde{\sigma}$ be its lifting on $\tilde{X}_\infty$. Then
\begin{align}
\tilde{\sigma} =  a \cdot dz_1\wedge dz_2 + O(|z|^{-3}\tilde{\sigma})
\end{align}
where $a = O(|z|^{\tau})$ as $r\to \infty$ is a $\bar\partial_{Euc}$-holomorphic function. 
Then $a = 0$, and $\tilde{\sigma} = O(|z|^{-3+\tau})$, which implies that $\tilde{\sigma}$ can decay to any order in the ALE coordinates $\{z_1,z_2\}$. Consequently, $\tilde{\sigma}$ can be extended to $C$. Furthermore, 
\begin{align}
\Omega^{2,0}|_C = \Omega^{2,0}(C)\otimes N_C^* = \oo(-3).
\end{align}
This implies that $\tilde{\sigma}$ vanishes on an open neighborhood of $C$, so $\sigma$ vanishes identically on $X$.

\end{proof}

\section{Small deformations of complex structure}
\label{Kuranishi}
First we need a fixed point theorem for operators on Banach spaces (see, for example \cite{BiquardRollin}).
\begin{lemma}
\label{l2.4}
Let $F: \mathcal{B}_1 \rightarrow \mathcal{B}_2$ be a bounded differentiable operator between Banach spaces. 
In a small neighborhood of $0\in \mathcal{B}_1$, $F(x) = F(0)+F'(0)\,x+Q(x)$, where 
\begin{align}
\label{Q}
\Vert Q(x)-Q(x') \Vert _{\mathcal{B}_2} \leq C_0\cdot (\Vert x \Vert_{\mathcal{B}_1} +\Vert x'\Vert_{\mathcal{B}_1})\cdot \Vert x -x'\Vert_{\mathcal{B}_1}.
\end{align}
Assume that $\Vert F(0)\Vert_{\mathcal{B}_2} \ll 1$. Then 

\noindent
(i) If $F'(0)$ is an isomorphism with right inverse $G$ bounded by $C_1$, then there exists a ball $B(0,s) \subset \mathcal{B}_1$ such that there exists a unique $x\in B(0,s)$ with $F(x)=0$.\\
(ii) If $F'(0)$ is Fredholm and surjective, with right inverse bounded by $C_1$, 
then there is an $s > 0$, so that $F^{-1}(0) \cap B(0,s)$ 
is isomorphic to $U \subset \ker(F'(0))$, where $U$ is a small neighborhood of the origin.
\end{lemma}
The following lemma gives the properties of the linearized operator we will 
need in order to invoke Lemma \ref{l2.4}. 
\begin{lemma} Let $(X,g,J)$ be a K\"ahler ALE surface. 
\label{l2.5}
The linear operator 
\begin{align}
P: C^{k,\alpha}_{\delta-1}(X,\Lambda^{0,1}\otimes\Theta)&\xrightarrow{(\bar\partial^*,\bar\partial)} C^{k-1,\alpha}_{\delta-2}(X,\Theta)\oplus C^{k-1,\alpha}_{\delta-2}(X,\Lambda^{0,2}\otimes\Theta)
\end{align}
is surjective and Fredholm, for some $k\geq 3$ and $\delta\in (-2,-1)$.
\end{lemma}
\begin{proof}
It is not hard to see that $P$ is an elliptic operator, 
and that the indicial roots of $P$ are integral. Consequently, by standard
weighted space theory, $P$ is Fredholm since $\delta$ is non-integral 
\cite{LockhartMcOwen}. 

 The cokernel is given by 
\begin{align}
\ker(P^*)=\{(\sigma_1,\sigma_2)\in C^{k-1,\alpha}_{-2-\delta}(\Theta)\oplus C^{k-1,\alpha}_{-2-\delta}(\Lambda^{0,2}\otimes\Theta):\bar\partial^*\sigma_2=\bar\partial\sigma_1\}.
\end{align}
Let $(\sigma_1, \sigma_2) \in \ker(P^*)$, then 
\begin{align}
\bar\partial^*\sigma_1=0,\; \bar\partial^*\bar\partial\sigma_1 = \bar\partial^*\bar\partial^*\sigma_2=0, 
\end{align}
which implies that $\sigma_1$ is $\square$-harmonic. From Proposition \ref{p2.5}, 
$\sigma_1=0$. Then 
\begin{align}
\bar\partial\sigma_2 = 0,\; \bar\partial^*\sigma_2 = 0,
\end{align}
and the vanishing of $\sigma_2$ follows from Proposition \ref{p2.6}.
Since $\ker(P^*)=0$, and $P$ is Fredholm, $P$ is surjective. 
\end{proof}

For a fixed complex structure $J$,
there is a nonlinear correspondence between sufficiently small sections of $\Gamma(X,End_a(TX))$ and almost complex structures near $J$ given by
\begin{align}
\begin{split}
\label{almostcomplex}
E_J: &\Gamma(X,End_a(TX))\rightarrow \mathcal{A}\\
E_J(I) &= \Big(J+\frac{1}{2}I \Big)J \Big(J+\frac{1}{2}I \Big)^{-1}, 
\end{split}
\end{align}
where $\mathcal{A}$ is the space of almost complex structures on $X$.
If $\phi \in \Gamma(\Lambda^{0,1} \otimes \Theta)$ has sufficiently small norm, 
then the corresponding almost complex structure will be denoted by
$J(\phi) = E_J(Re(\phi))$. Note that there is an expansion:
\begin{align}
\label{expansion}
J(\phi) = J + Re(\phi) + Q
\end{align}
where $|Q|= O (|\phi|^2)$, as $|\phi| \rightarrow 0$. 

Next we prove a weighted version of Kuranishi's theorem.
\begin{theorem} 
\label{kurthm}
Let $(X,J_0,g_0)$ be a K\"ahler ALE surface, and
let $B_{\epsilon_1}$ be an $\epsilon_1$-ball around $0$ in $\H_{-3}(X,\Lambda^{0,1}\otimes\Theta)$. Then for $\epsilon_1 > 0$ sufficiently small, there is a differentiable family of complex structures $J_t$ for $t \in B_{\epsilon_1}$ such that 
\begin{align}
J_t = J(\phi(t)), \ \phi(t) \in C^{k,\alpha}_{\delta-1}(X,\Lambda^{0,1}\otimes\Theta), \ \delta\in (-2,-1), \ k\geq 3,
\end{align}
where $\phi(t)$ satisfies 
$\phi(t) = t + \phi(t)^{\perp}$, 
with $\phi(t)^{\perp}$ is $L^2$-orthogonal to $\H_{-3}(X,\Lambda^{0,1}\otimes\Theta)$,
and 
$\bar\partial^* ( \phi(t)) = 0$.
Furthermore, 
there exists an $\epsilon' > 0$, such that if $J = J_0(\phi)$ 
is any integrable complex structure satisfying 
\begin{align}
\label{family_nbhd}
\Vert \phi \Vert_{C^{k,\alpha}_{\delta}} < \epsilon',
\mbox{ \ and \ } 
\bar\partial^* ( \phi) = 0,
\end{align} 
then $\phi$ is in the family $J_t$. Finally, for any $t \in B_{\epsilon_1}$, 
there exists a constant $C$ such that $\Vert \phi(t)\Vert_{C^{k,\alpha}_{-3}}\leq C\cdot\epsilon_1$.
\end{theorem} 
\begin{proof}
Define the operator 
\begin{align}
\begin{split}
F: C^{k,\alpha}_{\delta-1}(X,\Lambda^{0,1}\otimes\Theta)&\rightarrow C^{k-1,\alpha}_{\delta-2}(X,\Theta)\oplus C^{k-1,\alpha}_{\delta-2}(X,\Lambda^{0,2}\otimes\Theta)\\
\phi &\mapsto (\bar{\partial}^*\phi, \bar{\partial}\phi+[\phi,\phi]),
\end{split}
\end{align}
where $[ \phi, \phi]$ is a globally defined operator, 
which can be expressed locally as 
\begin{align}
[ \phi, \phi] &= \frac{1}{2}\Big[ \sum_{i,j} \phi_{i,j} d\bar{z}_i\otimes\frac{\partial}{\partial z_j}, \sum_{k,l} \phi_{k,l}d\bar{z}_k\otimes\frac{\partial}{\partial z_l}\Big] \\
&= \sum_{i,j,k,l} \phi_{i,j}\frac{\partial\phi_{k,l}}{\partial z_j} d\bar{z}_i\wedge d\bar{z}_k\otimes\frac{\partial}{\partial z_l}.
\end{align}
Formally, $[\phi,\phi]$ can be written as $\phi*\nabla\phi$, where $*$ means a linear combination of quadratic terms, each involving some contraction of $\phi$ with $\nabla\phi$.

Clearly, $F$ is a bounded differentiable mapping, and 
$F'(0) = (\bar{\partial}^*,\bar{\partial})$. Obviously,
in a small neighborhood of $0$ in 
$C^{k,\alpha}_{\delta-1}(X,\Lambda^{0,1}\otimes\Theta)$, $F$ admits an expansion 
$F = F(0)+F'(0)+Q$. We next prove the estimate \eqref{Q}.
Let $r$ denote the radius. For any two elements $\phi,\phi'\in C^{k,\alpha}_{\delta-1}(\Lambda^{0,1}\otimes \Theta)$, we have
\begin{align}
\begin{split}
r^{2-\delta} |[\phi,\phi]-[\phi',\phi']| &= r^{2-\delta}|[\phi*\nabla\phi]-[\phi'*\nabla\phi']|\\
&= r^{2-\delta}|\phi*(\nabla\phi-\nabla\phi')-(\phi-\phi')*\nabla\phi'|\\
&\leq C r^{\delta-1}\{r^{1-\delta} |\phi|r^{2-\delta}|\nabla\phi-\nabla\phi'|+ r^{1-\delta}|\phi-\phi'|r^{2-\delta}|\nabla\phi'|\}\\
& \leq C (r^{1-\delta} |\phi|) (r^{2-\delta} |\nabla(\phi-\phi')|)
+ (r^{2-\delta} |\nabla \phi'|) (r^{1-\delta} | \phi - \phi'|).
\end{split}
\end{align}
Next, let $x\neq y \in X$, with $r(x)<r(y)$. Similarly, we have
\begin{align}
\begin{split}
&r(x)^{2-\delta} \frac{|([\phi,\phi]-[\phi',\phi'])(x)-([\phi,\phi]-[\phi',\phi'])(y)|}{d(x,y)^\alpha} \\
&\leq r(x)^{\delta-1} \Big\{ r(x)^{1-\delta}\frac{|\phi(x)-\phi(y)|}{d(x,y)^\alpha} |\nabla\phi(x)-\nabla\phi'(x)| r(x)^{2-\delta}\Big\} +\\
& r(y)^{\delta-1} \Big\{ r(x)^{2-\delta}\frac{|(\nabla\phi(x)-\nabla\phi'(x))-(\nabla\phi(y)-\nabla\phi'(y))|}{d(x,y)^\alpha} |\phi(y)| r(y)^{1-\delta}\Big\} + \\
& r(x)^{\delta-1} \Big\{ r(x)^{1-\delta} \frac{|(\phi(x)-\phi'(x))-(\phi(y)-\phi'(y))|}{d(x,y)^\alpha} |\nabla\phi'(x)| r(x)^{2-\delta} \Big\}+\\
& r(y)^{\delta-1} \Big\{ r(x)^{2-\delta}\frac{|\nabla\phi'(x)-\nabla\phi'(y)|}{d(x,y)^\alpha} |\phi(y)-\phi'(y)| r(y)^{1-\delta} \Big\}.
\end{split}
\end{align}
This shows that
\begin{align}
\Vert [\phi,\phi]-[\phi',\phi'] \Vert_{C^{0,\alpha}_{\delta-2}}\leq C_0 (\Vert \phi\Vert_{C^{1,\alpha}_{\delta-1}}+\Vert \phi' \Vert_{C^{1,\alpha}_{\delta-1}})\Vert \phi-\phi'\Vert_{C^{1,\alpha}_{\delta-1}}.
\end{align}
The higher derivative terms can be handled similarly, to prove
\begin{align}
\Vert [\phi,\phi]-[\phi',\phi'] \Vert_{C^{j-1,\alpha}_{\delta-2}}\leq C_j (\Vert \phi\Vert_{C^{j,\alpha}_{\delta-1}}+\Vert \phi' \Vert_{C^{j,\alpha}_{\delta-1}})\Vert \phi-\phi'\Vert_{C^{j,\alpha}_{\delta-1}} ,\; 1\leq j\leq k
\end{align}
which shows that $Q$ satisfies \eqref{Q}. 

By Lemma $\ref{l2.5}$, we know $P = F'(0)$ is Fredholm. We can choose a right inverse operator $G$ such that the image of $G$ is $L^2$-orthogonal to $\H_{-3}(X,\Lambda^{0,1}\otimes\Theta)$. By Lemma \ref{l2.4} there exists $s = \epsilon_1 > 0$, such that $F^{-1}(0)$ is locally isomorphic to an $\epsilon_1$-neighborhood of $0$ in 
$\H_{\delta-1}(X,\Lambda^{0,1}\otimes\Theta)$, with the property that for each $t\in B_{\epsilon_1}$, there is a $\phi(t) = t+\phi(t)^{\perp}\in F^{-1}(0)$, where $\phi(t)^{\perp}$ is $L^2$-orthogonal to $\H_{-3}(X,\Lambda^{0,1}\otimes\Theta)$, and $J_0(\phi(t))$ is the corresponding complex structure. 
It is a straightforward consequence of the implicit function theorem that the mapping $\psi: t \mapsto \phi(t)$ is differentiable.  This finishes the proof of the existence of the family of complex structures.

By the fixed point Lemma \ref{l2.4}, near $0$, the zero set of $(\bar\partial^*,\bar\partial + [\,,\,])$ is locally bijective to the kernel of $(\bar\partial^*,\bar\partial)$, which is $\H_{-3}(X,\Lambda^{0,1}\otimes\Theta)$. Thus there exists an $\epsilon''>0$ such that for any $\Vert \phi\Vert_{C^{k,\alpha}_{\delta-1}}<\epsilon''$, $\bar\partial^*\phi=0$, $\bar\partial\phi+[\phi,\phi]=0$, then $J = J_0(\phi)$ is in the family we just constructed. 

Next we will show that $\phi(t)\in C^{k,\alpha}_{-3}(X,\Lambda^{0,1}\otimes\Theta)$ for $t\in B_{\epsilon_1}$. 
In fact, we will show that 
for any $\phi\in C^{k,\alpha}_{\delta}(X,\Lambda^{0,1}\otimes\Theta)$ that satisfies the system
\begin{align}
\label{condition}
\bar\partial\phi = -[\phi,\phi],\; \bar\partial^*\phi = 0,
\end{align}
then $\phi\in C^{k,\alpha}_{-3}(X,\Lambda^{0,1}\otimes\Theta)$. To see this, \eqref{condition} implies that 
\begin{align}
\label{higherOrders}
\|\square\phi\|_{C^{k-2,\alpha}_{2\delta-2}} = \|-\bar\partial^*[\phi,\phi]\|_{C^{k-2,\alpha}_{2\delta-2}}\leq C\cdot 
\|\phi\|^2_{C^{k,\alpha}_\delta}
\end{align}
Then outside of a compact subset,
\begin{align}
\label{lap}
\frac{1}{2}\Delta_{Euc}\phi = -\square\phi+(\frac{1}{2}\Delta_{Euc}+\square)\phi = O(r^{2\delta-2}),
\end{align}
which implies that 
\begin{align}
\phi = \sum_{i,j}\frac{a_{i,j}}{r^2}d\bar{z_i}\otimes\frac{\partial}{\partial z_j}+O(r^{-3+\epsilon}),
\end{align}
 where $a_{i,j}$ are constants and $\frac{a_{i,j}}{r^2}$ is $\Delta_{Euc}$-harmonic. 
By the same argument as in the proof of \eqref{aa1}, we have $a_{i,j}=0$. Then a 
similar argument shows that $\phi\in C^{k,\alpha}_{-3}(X,\Lambda^{0,1}\otimes\Theta)$. 

Next we will show that, if $\phi\in C^{k,\alpha}_{\delta}(X,\Lambda^{0,1}\otimes\Theta)$, $\|\phi\|_{C^{k,\alpha}_{\delta}}<\epsilon'$,
and satisfies the system \eqref{condition}, then $\|\phi\|_{C^{k,\alpha}_{\delta-1}}\leq C\cdot \epsilon'$ for some constant $C>0$.
By the argument above, $\phi\in C^{k,\alpha}_{-3}(X,\Lambda^{0,1}\otimes\Theta)$, and $\phi = \phi^h + \phi^{\perp}$, 
where $\phi^h\in \H_{-3}(X,\Lambda^{0,1}\otimes\Theta)$, and $\phi^h$ is $L^2$-orthognal to $\phi^{\perp}$. 
By an argument similar to Corollary 1.16 in \cite{Bartnik}, $\|\phi^{\perp}\|_{C^{k,\alpha}_{\delta-1}}\leq C\cdot \|\square\phi\|_{C^{k-2,\alpha}_{\delta-3}}$. 
Since $\phi^h$ is in a finite dimensional space, combined with the estimate \eqref{higherOrders},
we can make $\epsilon'$ small enough such that $\|\phi\|_{C^{k,\alpha}_{\delta-1}}<\epsilon''$, and then $\phi$ is in
the family $J_t$ as shown above.
The estimate of $\|\phi\|_{C^{k,\alpha}_{-3}}$ follows from finite-dimensionality of the kernel. 
\end{proof}

\section{Versality of Kuranishi family}
\label{vkf}
The next result shows that any nearby complex structure can 
be brought into the family $J_t$ by a suitable diffeomorphism. 
Of course, the implicit function theorem requires a mapping between 
Banach spaces to be differentiable. The following lemma, inspired by \cite{Biquard06}, will be used below to this end. 
\begin{lemma}
\label{smooth}
Let $(X,g_0,J_0)$ be a K\"ahler ALE surface with $g_0,J_0\in C^{\infty}$. 
Then for $\delta'<-1$, for $g_0, J_0$ with ALE asymptotic rate $O(r^{\delta'})$,
there exists an $\epsilon_1>0$, such that for an $\epsilon_1$-ball $B_{\epsilon_1}\subset C^{k+1,\alpha}_{\delta'+1}(TX)$, the maps 
\begin{align}
B_{\epsilon_1}&\rightarrow C^{k,\alpha}_{\delta'}(S^2(T^*X)),\;
Y \mapsto {\Phi_Y}_* g_0-g_0\\
B_{\epsilon_1}&\rightarrow C^{k,\alpha}_{\delta'}(End(TX)),\;
Y \mapsto {\Phi_Y}_* J_0-J_0
\end{align}
are smooth. 
\end{lemma}
\begin{proof} 
Let $U\subset X$ be a normal coordinate chart with respect to $g_0$,
and let $U', U''$ be smaller charts, such that the $\epsilon_1$-neighborhood of $U'$ is in $U$, and the $\epsilon_1$-neighborhood of $U''$ is in $U'$.
The tangent bundle $TU$ is isomorphic to $U\times \RR^4$, with coordinate functions $(x,v) = (x_1,\ldots,x_4,v_1,\ldots v_4)$.
On $U'$, we have the Riemannian exponential map:
\begin{align}
\begin{split}
Exp: U'\times B_{\epsilon_1}\subset TU'&\rightarrow U\\
(x,v)&\rightarrow \gamma(x,v)
\end{split}
\end{align}
where $\gamma(x,v)$ is a geodesic with initial position and derivative $(x,v)$.
The $j^{\text{th}}$-coordinate function of $Exp$ $(1\leq j\leq 4)$ 
is a smooth function over $(x,v)$, and has an expansion
\begin{align}
\label{Expexpansion}
Exp_j(x,v) = x_j + \sum_{k=1}^4 f_k(x)v_k + Q(x,v)
\end{align}
where $|Q(x,v)|\leq C(U,g_0)\cdot |v|^2$. 
Let $\xi\in C^{k+1,\alpha}(TU)$, and $\|\xi\|_{C^{k+1,\alpha}}<\epsilon_1$.
The geodesic flow $\Phi_{\xi}$ is defined as
\begin{align} 
\begin{split}
\Phi_{\xi}: U'&\rightarrow U\\
x &\rightarrow Exp(x,\xi(x))
\end{split}
\end{align} 
Then $\Phi_{\xi}(x) = Exp(x,\xi(x))$ is $C^{k+1,\alpha}$ over $x$, since $Exp$ is smooth and $\xi(x)$ is $C^{k+1,\alpha}$.
Let $\xi'$ be another vector field on $U$ with $\|\xi'\|_{C^{k+1,\alpha}}<\epsilon_1$.
Then by the expansion above, for each $1\leq j\leq 4$, we have:
\begin{align}
\label{Flowexpansion}
\|({\Phi_{\xi}})_j - ({\Phi_{\xi'}})_j - \sum_{k=1}^4 f_k\cdot (\xi_k-\xi'_k)\|_{C^{k+1,\alpha}(U'')}
\leq C\cdot \|\xi-\xi'\|^2_{C^{k+1,\alpha}(U)}
\end{align}
This implies that each coordinate function of $\Phi$ is Fr\'echet differentiable as a map from $B_{\epsilon_1}\subset C^{k+1,\alpha}(TU)$ to 
$C^{k+1,\alpha}(U'')$.

Since ${\Phi_\xi}_*g_0$ is defined as $(\Phi_{\xi}^{-1})^*g_0$, without the loss of generality, we will do the analysis on $\Phi^*_\xi g_0-g_0$.
On the chart $U''$, we have 
$\Phi^*_\xi g_0(x) (\partial_i,\partial_j) = g_0(\Phi_\xi(x))({\Phi_\xi}_*\partial_i, {\Phi_\xi}_*\partial_j)$. 
Since $g_0$ is smooth over $x$, and ${\Phi_\xi}_*\partial_j = \partial_j{\Phi_\xi}$ is $C^{k,\alpha}$ over $x$, by using the argument above, 
we can see that 
$\Phi^*g_0-g_0: B_{\epsilon_1}\subset C^{k+1,\alpha}(TU)\rightarrow C^{k,\alpha}(S^2(T^*U''))$ is Fr\'echet differentiable. 

For some large $R>0$, denote $A_{R,2R}$ as an annulus in the ALE space $X$. By covering $A_{R,2R}$ with finite many normal coordinate charts,
using the argument above for each chart, we have the map 
$\Phi^*g_0: B_{\epsilon_1}\subset C^{k+1,\alpha}(TA_{R,2R})\rightarrow C^{k,\alpha}(S^2(T^*A_{R-\epsilon_1,2R-2\epsilon_1}))$
is Fr\'echet differentiable.
By a standard dilation argument, which dilates $A_{R,2R}$ to $A_{2^kR,2^{k+1}R}$, we have
$\Phi^*g_0 - g_0: B_{\epsilon_1}\subset C^{k+1,\alpha}_{\delta'+1}(TA_{2^kR,2^{k+1}R})\rightarrow C^{k,\alpha}_{\delta'}(S^2(T^*A_{2^k(R-\epsilon_1),2^{k+1}(R-\epsilon_1)}))$
is Fr\'echet differentiable. The constant $C(U,g_0)$ in each chart $U$ can be chosen uniformly bounded on $X$. This implies that
$\Phi^* g_0 - g_0: B_{\epsilon_1}\subset C^{k+1,\alpha}_{\delta'+1}(TX)\rightarrow C^{k,\alpha}_{\delta'}(S^2(T^*X))$ is Fr\'echet differentiable. Furthermore, it follows that
for $Y,Y'\in B_{\epsilon_1}$, there exists a constant $C$,
\begin{align}
\|\Phi_{Y}^*g_0 - \Phi_{Y'}^*g_0 - \mathfrak{L}_{Y-Y'}g_0\|_{C^{k,\alpha}_{\delta'}} \leq C\cdot \|Y-Y'\|_{C^{k+1,\alpha}_{\delta'+1}}\cdot (\|Y\|_{C^{k+1,\alpha}_{\delta'+1}} +\|Y'\|_{C^{k+1,\alpha}_{\delta'+1}}).
\end{align}
By using higher order expansions similar to \eqref{Expexpansion}, we can furthermore show that $\Phi^*g_0 - g_0$ is indeed a smooth map. The proof for $J_0$ can proceed in a similar fashion. Specifically we have
\begin{align}
\label{quadratic}
\|\Phi_{Y}^*J_0 - \Phi_{Y'}^*J_0 - \mathfrak{L}_{Y-Y'}J_0\|_{C^{k,\alpha}_{\delta'}} \leq C\cdot \|Y-Y'\|_{C^{k+1,\alpha}_{\delta'+1}}
\cdot( \|Y\|_{C^{k+1,\alpha}_{\delta'+1}} + \|Y'\|_{C^{k+1,\alpha}_{\delta'+1}}).
\end{align}
\end{proof}

Let $\nabla^* : \Gamma( End_{a}(TX) ) \rightarrow \Gamma(TX)$ be 
the adjoint operator of $\frac{-1}{2}J\circ \mathfrak{L}_*J$, which is defined as: 
\begin{align}
(\nabla^* A)^{k} = - \sum_{l,j} g^{lj} \nabla_l A^k_j. 
\end{align}
Note that under the identification of $End_a(TX)$ with $Re(\Lambda^{0,1} \otimes \Theta)$
and $TX$ with $Re(\Theta)$, this corresponds to the operator $\overline{\partial}^*$. 
In the following proofs, for the convenience of the notation, we will denote each element in $\Gamma(End_a(TX))$ as $Re(\phi)$, where 
$\phi\in \Gamma(\Lambda^{0,1}\otimes\Theta)$.

\begin{lemma}
\label{smoothop}
Let $(X,J_0,g_0)$ be a K\"ahler ALE surface with $J_0,g_0\in C^\infty$. Let $\delta'<-1, k\geq 3, 0<\alpha<1$.
Then the following map $\mathfrak{P}$ is smooth in an open neighborhood of $(0,0,0)$
\begin{align}
\label{smoothBanach}
\begin{split}
\mathfrak{P}: C^{k,\alpha}_{\delta'}(S^2(TX))\times C^{k,\alpha}_{\delta'}(End(TX))\times C^{k,\alpha}_{\delta'}(End_a(TX)) &\rightarrow 
C^{k-1,\alpha}_{\delta'-1}(TX) \\
(h, w, Re(\phi))&\mapsto \nabla^*_{g_0+h}(E^{-1}_{J_0+w}(J_0(\phi)))
\end{split}
\end{align}
\end{lemma}
\begin{proof}
The divergence operator has the expansion formula
\begin{align}
\nabla_{g_0+h} = \nabla_{g_0} + (g_0+h)^{-1}*\nabla_{g_0} h, 
\end{align}
\cite[Formula 3.39]{GurskyViaclovsky16}, 
where $*$ means a linear combination of tensor contractions, and $(g_0+h)^{-1}$ is analytic for smooth $h$.
By the expansions \eqref{expansion}, 
\begin{align}
\label{expansionE}
E^{-1}_{J_0+w}(J_0(\phi)) = -w + Re(\phi)+Q'
\end{align}
where $Q'$ represents the higher order terms, which is analytic for small $w$.
Note that $\mathfrak{P}(0,0,0) = 0$. 
For $(h,w,Re(\phi)), (h',w',Re(\phi'))$ of sufficiently small norm, by the expansion above, there exists a constant $C$ so that 
\begin{align}
\begin{split}
\|\mathfrak{P}(h,w,Re(\phi))- \mathfrak{P}(h',w',Re(\phi'))
+ \nabla^*_{g_0}(w-w') - \nabla^*_{g_0}Re(\phi-\phi') \|_{C^{k-1,\alpha}_{\delta'-1}} \leq \\
C\cdot (\|h\|_{C^{k,\alpha}_{\delta'}}+ \|h'\|_{C^{k,\alpha}_{\delta'}}+
\|w\|_{C^{k,\alpha}_{\delta'}}+  \|w'\|_{C^{k,\alpha}_{\delta'}} + 
\|Re(\phi)\|^2_{C^{k,\alpha}_{\delta'}} + \|Re(\phi')\|^2_{C^{k,\alpha}_{\delta'}} ) \\
\cdot(\|h-h'\|_{C^{k,\alpha}_{\delta'}}+ \|w-w'\|_{C^{k,\alpha}_{\delta'}} + \|Re(\phi-\phi')\|^2_{C^{k,\alpha}_{\delta'}})
\end{split}
\end{align}
This implies $\mathfrak{P}$ is differentiable, and satisfies the condition of Lemma \ref{l2.4}. 
By higher order expansions of $\mathfrak{P}$ over $(h,w,Re(\phi))$, we can furthermore
show that $\mathfrak{P}$ is smooth.
\end{proof}

\begin{lemma}
\label{dfree}
Let $(X,J_0,g_0)$ be a K\"ahler ALE surface with $J_0,g_0 \in C^{\infty}$. There exists an $\epsilon_1'>0$ such that for any complex structure $\Vert J_1-J_0\Vert_{C^{k,\alpha}_{\delta}}<\epsilon_1'$, where $k\geq 3, \alpha\in (0,1),\delta\in (-2,-1)$, there exists a unique diffeomorphism $\Phi$,
of the form $\Phi_Y$ for $Y \in C^{k+1,\alpha}_{\delta+1}(TX)$
such that 
$\Phi^*(J_1)$ is in the family $J_t$ from Theorem \ref{kurthm}.
\end{lemma}
\begin{proof}
Let $J_1 = J_0(\phi)$,
where $\Vert \phi\Vert_{C^{k,\alpha}_{\delta}}<\epsilon_1'$ for some small $\epsilon_1'$ to be determined later, and $\bar\partial\phi+[\phi,\phi]=0$. 
Define $\phi_Y \in \Gamma(\Lambda^{0,1} \otimes \Theta)$ by 
$\Phi_Y^* J_0(\phi) = E_{J_0}(Re(\phi_Y))$. 
Define the operator $\mathfrak{N}$ by choosing the first parameter of $\mathfrak{P}$ to be $(\Phi_Y)_*g_0-g_0$, choosing
the second parameter to be $(\Phi_Y)_*J_0-J_0$, and letting $\delta' = \delta$ as the following
\begin{align}
\begin{split}
C^{k+1,\alpha}_{\delta+1}(TX)\times C^{k,\alpha}_{\delta}(X,End_a(TX)) &\rightarrow C^{k-1,\alpha}_{\delta-1}(TX)\\
(Y,Re(\phi)) &\mapsto \nabla^*_{{\Phi_Y}_*g_0}(E^{-1}_{{\Phi_Y}_*J_0}(J_0(\phi))),
\end{split}
\end{align}
where $\Phi_Y$ is defined in \eqref{phidef}, 
and $\nabla^*_{{\Phi_Y}_* g_0}$ is the divergence operator associated to metric ${\Phi_Y}_*g_0$. 
Note that 
\begin{align}
\label{trick}
\nabla^*_{{\Phi_Y}_*g_0}(E^{-1}_{{\Phi_Y}_*J_0}(J_0(\phi))) = 0  \Leftrightarrow \nabla^*_{g_0} E_{J_0}^{-1}(\Phi_Y^*(J_0(\phi)))) = 0 
\Leftrightarrow \nabla^*_{g_0} Re(\phi_Y) = 0,
\end{align}
so a zero of $\mathcal{N}$ satisfies the desired gauge condition. 
(The latter map has a regularity issue which is why we consider the former, see \cite{Biquard06}.) 
Since $\mathfrak{P}$ is smooth by Lemma \ref{smoothop},  
and since the mappings $Y \mapsto (\Phi_Y)_*g_0 - g_0$, and $Y \mapsto (\Phi_Y)_*J_0-J_0$ are smooth by Lemma \ref{smooth}, the composition
$\mathfrak{N}$ is smooth in an open neighborhood of $(0,0)$. 

Consequently, at $(0,0)$, the linearization of $\mathfrak{N}$ restricted to the tangent space of the first variable is 
\begin{align}
\begin{split}
& D\,\mathfrak{N}_0: C^{k+1,\alpha}_{\delta+1}(TX) \rightarrow C^{k-1,\alpha}_{\delta-1}(TX) \\
&Y \mapsto \nabla^*_{g_0}{\circ \frac{-1}{2}J\circ}\mathfrak{L}_YJ_0 =  \square_{g_0} Y, 
\end{split}
\end{align}
where the adjoint operator of $D\,\mathfrak{N}_0$ is:
\begin{align}
\begin{split}
(D\,\mathfrak{N}_0)^*: C^{k+1,\alpha}_{-3-\delta}(TX)&\rightarrow C^{k-1,\alpha}_{-5-\delta}(TX) \\
\eta &\mapsto \square_{g_0} \eta, 
\end{split}
\end{align}
The kernel of $(D\,\mathfrak{N}_0)^*$ consists of  those $\eta\in C^{k+1,\alpha}_{-3-\delta}(TX)$ such that $\square_{g_0} \eta=0$, which implies that 
$\eta$ is the real part of a holomorphic vector field. 
From Proposition \ref{p2.5}, $\eta=0$. Since the cokernel of $D\,\mathfrak{N}_0$ is trivial, $D\,\mathfrak{N}_0$ is surjective. Note also that the kernel of $D\,\mathfrak{N}_0$ is trivial. 

Of course, in a small neighborhood of $(0,0)$ in $C^{k+1,\alpha}_{\delta+1}(TX)\times C^{k,\alpha}_\delta(X,End_a(TX))$, $D\,\mathfrak{N}$ is surjective.
Since $\mathfrak{N}$ is smooth, condition \eqref{Q} is satisfied, so by Lemma \ref{l2.4}, there exists a unique diffeomorphism $\Phi_Y$ such that 
$Re(\phi_Y)$ is divergence-free.
Since {$\phi_Y$} satisfies \eqref{family_nbhd}, by Theorem~\ref{kurthm}, $\Phi^*_Y(J_1)$ is in the family~$J_t$.
\end{proof}

\section{Essential deformations of complex structure}
\label{esssec}
The arguments in this section are inspired by \cite[Theorem 3.1]{ct}, with 
some appropriate modifications for the smoothness arguments.
Recall that the space $\mathbb{W}$ is defined as
\begin{align}
\mathbb{W} = \{ Z \in \mathcal{H}_1(X, TX) \ | \ \mathfrak{L}_Z g_0  = O (r^{-1}),
\ \mathfrak{L}_Z J_0 = O(r^{-3}), \mbox{ as } r \rightarrow \infty\}.
\end{align}
Clearly, $\WW$ is 
finite-dimensional. For $Z \in \mathbb{W}$, we define $\Upsilon_Z : X \rightarrow X$ 
to be time-one flow for the one-parameter group of diffeomorphisms 
generated by $Z$. Note that since $\mathbb{W}$ is finite-dimensional the use of the one-parameter group here will not have the loss-of-regularity problem as in infinite dimensions. Furthermore, the time-one flow for vector fields in $\mathbb{W}$ has the following nice property.  
\begin{lemma}  
\label{Zlem}
Fix an ALE coordinate under which $|J_0-J_{Euc}| = O(r^{-3})$.
For $Z \in \WW$ sufficiently small, the  
time-one flow for the one-parameter group of diffeomorphisms generated by $Z$
exists. Furthermore, there exist a matrix $A \in {\rm{U}}(2)$ and vector 
$\hat{\epsilon}_0 \in \CC^2$, such that in ALE coordinates of order $\mu$ with 
$-2 < -\mu < -1$, 
\begin{align}
d( \Upsilon_Z(x), Ax + \hat{\epsilon}_0) = O(r^{-\mu}),
\end{align}
as $r \rightarrow \infty$, for any $\epsilon > 0$. 
\end{lemma}
\begin{proof}

First, note that $Z \in \WW$ admits an expansion 
\begin{align}
\label{Zexp}
Z = Z_1 + Z_0 + Z_{\mu}, 
\end{align}
for any $\epsilon > 0$,  where $Z_1$ is in the Lie algebra of ${\rm{U}}(2)$,
$Z_0$ is vector field with constant coefficients,
and $Z_{\mu} =  O(r^{-\mu})$ as $r \rightarrow \infty$. 
To see this, by standard harmonic expansion for kernel 
elements of $\square_{g_0}$, we have that any $Z \in \WW$ satisfies
\begin{align}
Z = Z_1 + Z_0 + O(r^{-\mu}),
\end{align}
where $Z_1$ is any vector field which is homogeneous of degree one, 
and $Z_0$ is a vector field with constant coefficients. 
The conditions that $\mathfrak{L}_Z g_0 = O(r^{-1})$ 
and $\mathfrak{L}_Z J_0 = O(r^{-3})$ as $r \rightarrow \infty$ imply 
that $\mathfrak{L}_{Z_1} g_{Euc} = 0$ and $\mathfrak{L}_{Z_1} J_{Euc} = 0$,
that is, $Z_1$ is in the Lie algebra of ${\rm{U}}(2)$. Next, 
\begin{align}
\mathfrak{L}_{Z_1 + Z_0} J_0 
= \mathfrak{L}_{Z_1 + Z_0} (J_0 - J_{Euc})  
=  \mathfrak{L}_{Z_1 + Z_0} (O(r^{-3})),
\end{align}
which implies that 
\begin{align}
\mathfrak{L}_{Z_0 + Z_1} J_0  = O(r^{-3}),
\end{align}
as $r \rightarrow \infty$. Consequently, 
\begin{align}
\square_0 ( Z - Z_1 - Z_0) = O(r^{-4}). 
\end{align}
Since there is no harmonic term of degree $-1$, the claimed
expansion \eqref{Zexp} follows. 

Clearly for $Z$ sufficiently small in norm, the one-parameter group of diffeomorphisms
which is defined by 
\begin{align}
x'(t) = Z(x(t)), \; x(0) = x,
\end{align}
exists by the standard short-time existence theorem for ODEs.
By the expansion \eqref{Zexp}, $x(t)$ will be given by a family of rotations
in ${\rm{U}}(2)$, and translations plus decaying terms.
\end{proof}
\begin{remark}{\em
Lemma \ref{Zlem} has assumed an ALE coordinate with inverse cubic complex structure asymptotic rate. In the following, note that the statements Lemma \ref{W} and Theorem \ref{t5.2} are independent of the ALE coordinate in the sense that, if they are true under the ALE coordinate $(x^1,x^2,x^3,x^4)$, then they are
also true under the ALE coordinate $({x'}^1,{x'}^2,{x'}^3,{x'}^4)$, as long as the two coordinates
are comparable, i.e, $\frac{1}{C}|x|<|x'|<C|x|$ for some constant $C>0$.
Thus, in their proofs, we are free to choose an ALE coordinate such that $|J_0-J_{Euc}|=O(r^{-3})$ by Proposition \ref{p1.4}.
}
\end{remark}

The next lemma states some useful properties of $\mathbb{W}$, which will be 
needed in the proof of the subsequent gauging theorem.  
For the simplicity of the statement of the theorem, define the norm $\|Z\|_{k+1,\alpha}$ for $Z\in \mathbb{W}$ as
$\|Z\|_{k+1,\alpha} = \|Z\|_{C^{k+1,\alpha}_1}$ if $\lim_{r\to\infty} |r^{-1}\cdot Z| \neq 0$,
$\|Z\|_{k+1,\alpha} = \|Z\|_{C^{k+1,\alpha}}$ if $\lim_{r\to\infty} |r^{-1}\cdot Z| = 0$.
Note that if $Z \in \mathbb{W}$ is decaying, then, $Z \equiv 0$.
\begin{lemma}
\label{W}
Let $(X,J_0,g_0)$ be a K\"ahler ALE surface, the space $\mathbb{W}$ satisfies the following properties: 
\begin{itemize}
\item
There exists a small neighborhood $U\subset \mathbb{W}$ of $0$, 
and an $\epsilon' > 0$ such that if $g_1$ is any metric
satisyfing 
$\Vert g_1 - g_0 \Vert_{C^{k,\alpha}_{\delta}} < \epsilon'$, 
then  there exists a constant $C$, such that for any $Z\in U$, 
\begin{align}
\label{wl2}
\Vert \Upsilon^*_Z g_1 - g_0 \Vert_{C^{k,\alpha}_{\delta}} 
\leq C ( \Vert Z\Vert_{k+1,\alpha} + \epsilon').
\end{align}
\item
Given $\epsilon >0$ sufficiently small, for any complex structure $J_1 = J_0(\phi)$ such that $\Vert \phi \Vert_{C^{k,\alpha}_{-3+ \epsilon }}<\epsilon'$ 
in the family constructed in Theorem \ref{kurthm}
for some sufficiently small $\epsilon'$, there exists a small neighborhood $U\subset \mathbb{W}$ of $0$, and a constant $C$, such that for any $Z,Z'\in U$,
\begin{align}
\label{wl3}
\Vert \Upsilon^*_Z J_1 - \Upsilon^*_{Z'}J_1 -  \mathfrak{L}_{Z-Z'} J_1 \Vert_{C^{k,\alpha}_{-3 + \epsilon}} 
\leq C \cdot \Vert Z-Z'\Vert_{k+1,\alpha}\cdot(\Vert Z\Vert_{k+1,\alpha}+\Vert Z'\Vert_{k+1,\alpha}).
\end{align}
Furthermore, the following map 
\begin{align}
\label{Upsilon}
\begin{split}
\Upsilon^*: \H_{-3}(X,End_a(TX))\times \mathbb{W} &\rightarrow C^{k,\alpha}_{-3+\epsilon}(X,End_a(TX)) \\
(Re(t), Z) & \mapsto E_{J_0}^{-1}(\Upsilon^*_Z J_0(\phi(t))) \equiv \phi_Z(t)
\end{split}
\end{align}
is smooth in an open neighborhood of $(0,0)$.
\end{itemize}
\end{lemma}
\begin{proof}
Fix an ALE coordinate under which $|J_0-J_{Euc}|=O(r^{-3})$.
Under such a coordinate, we have the decomposition of $Z$ as stated in Lemma \ref{Zlem}.
For simplicity, we assume $Z_1 \neq 0$ in the following,
and  $\|Z\|_{k+1,\alpha} = \|Z\|_{C^{k+1,\alpha}_1}$. (The proof for the case $Z_1 = 0$ is similar.)

First, we address the estimate \eqref{wl2}.
Let $\Upsilon_{A, \hat{\epsilon}_0}$ denote the mapping $x \mapsto A x +  \hat{\epsilon}_0$.
Using Lemma~\ref{Zlem},
\begin{align*}
|\Upsilon_Z^* g_1 - g_0 |
&= |\Upsilon_Z^* g_1 - g_1 + g_1 - g_0 |
\leq |\Upsilon_Z^* g_1 - g_1| + |g_1 - g_0 |\\
&
\leq |  ( \Upsilon_{A, \hat{\epsilon}_0} + O (r^{-2 + \epsilon}))^* ( \delta_{ij} + O(r^{\delta}))
 -( \delta_{ij} + O(r^{\delta})) | + \epsilon' r^{\delta}\\
& \leq C \Vert Z \Vert_{C^{k+1,\alpha}_{1}}  r^{\delta}   + \epsilon' r^{\delta} + o(r^{\delta})\\
& \leq C r^{\delta}( \Vert Z\Vert_{C^{k+1,\alpha}_{1}} + \epsilon').
\end{align*}
as $r \rightarrow \infty$. Since $\Upsilon_Z$ is the time one flow of $Z \in \WW$, 
and $\WW$ is finite-dimensional, a similar estimate holds on any 
compact subset of $X$, so the $C^0_\delta$ part of the norm is bounded by the 
right hand side of \eqref{wl2}.  
Higher regularity estimates are similar, and are omitted. 

Next we will discuss \eqref{wl3}. 
As in the first part, since $\WW$ is finite-dimensional, 
we only need to make estimates outside of a compact set $\overline{B_R(p_0)}$, where a global coordinate exists.
Let $\gamma:[0,1]\rightarrow X$ be the path $\gamma(t) = \Upsilon_{tZ}(x)$. 
First, we estimate
\begin{align}
\begin{split}
r(x)^{3-\epsilon}|\Upsilon^*_Z J_0-J_0-\mathfrak{L}_{Z}J_0|(x) 
&\leq C r^{3-\epsilon}\int_0^1 |\Upsilon^*_{tZ}(\mathfrak{L}_{Z}J_0) (\gamma(0))-\mathfrak{L}_{Z} J_0 (\gamma(0))| dt\\
&\leq C r^{3 - \epsilon}\int_0^1\int_0^t|\mathfrak{L}_{Z}(\mathfrak{L}_{Z}J_0) (\gamma(s))| ds  dt.
\end{split}
\end{align}
By the expansion \eqref{Zexp}, given $c>0$, the estimates 
\begin{align}
(1 - c) r(\gamma(s))\leq r(x) = r(\gamma(0))\leq (1+ c) r(\gamma(s))
\end{align} 
are satisfied for all $Z$ sufficiently small in norm.
Note that since $Z_1$ is a vector field corresponding to a rotation in $\mathfrak{u}(2)$, $|Z_1|\geq C\cdot r$ for some constant $C$.
Also, since $\mathbb{W}$ is finite dimensional, $|\nabla^{\mathcal{I}}Z|\leq C(|\mathcal{I}|)\cdot r^{-|\mathcal{I}|}\cdot (|Z|+1)$
for some constant $C(|\mathcal{I}|)>0$.
Then
\begin{align}
|\mathfrak{L}_{Z}(\mathfrak{L}_{Z}J_0)|(\gamma(s)) \leq Cr(\gamma(s))^{-5}( |Z(\gamma(s))|+1)^2,
\end{align}
so we have
\begin{align}
\begin{split}
r(x)^{3-\epsilon}|\Upsilon^*_Z J_0-J_0-\mathfrak{L}_{Z}J_0|(x) 
&\leq C \int_0^1\int_0^t r(\gamma(s))^{3-\epsilon} r(\gamma(s))^{-5}( |Z|+1)^2(\gamma(s))  ds dt\\
&\leq C \int_0^1\int_0^t r(\gamma(s))^{-2 - \epsilon}(|Z|+1)^2(\gamma(s)) ds dt \leq C\Vert Z\Vert_{C^{k+1,\alpha}_{1}}^2.
\end{split}
\end{align}
Next, we estimate
\begin{align}
\begin{split}
r(x)^{3-\epsilon}|\Upsilon^*_Z \phi-\phi-\mathfrak{L}_Z \phi|(x) 
& \leq C r^{3-\epsilon}\int_0^1 |\Upsilon^*_{tZ}(\mathfrak{L}_{Z}\phi)(\gamma(0))-\mathfrak{L}_Z\phi(\gamma(0))| dt\\
&\leq C r^{3-\epsilon}\int_0^1 \int_0^t|\mathfrak{L}_Z(\mathfrak{L}_{Z}\phi)(\gamma(s))|ds  dt.
\end{split}
\end{align}
Since $|\mathfrak{L}_{Z}(\mathfrak{L}_{Z}\phi)|(\gamma(s)) \leq  C r(\gamma(s))^{-5 + \epsilon}( |Z|+1)^2(\gamma(s))$, 
similarly, we have the estimate
\begin{align}
\begin{split}
r(x)^{3-\epsilon}|\Upsilon^*_Z \phi-\phi-\mathfrak{L}_Z \phi|(x) 
&\leq C \int_0^1\int_0^t r(\gamma(s))^{3-\epsilon}(r(\gamma(s))^{-5+\epsilon}( |Z|+1)^2(\gamma(s))) ds dt\\
&\leq C \int_0^1\int_0^t r(\gamma(s))^{-2}(|Z|+1)^2(\gamma(s)) ds dt \leq C\Vert Z\Vert_{C^{k+1,\alpha}_{1}}^2.
\end{split}
\end{align}
By a similar calculation, for the higher order term $Q$ in \eqref{expansion}  we also have
\begin{align}
r(x)^{3-\epsilon}|\Upsilon^*_Z Q-Q-\mathfrak{L}_Z Q|(x) \leq C\Vert Z\Vert_{C^{k+1,\alpha}_{1}}^2
\end{align}
It follows that
\begin{align}
\label{estimateR}
\begin{split}
r(x)^{3-\epsilon}|\Upsilon^*_Z J_1-J_1-\mathfrak{L}_{Z} J_1|(x) 
&\leq r(x)^{3-\epsilon}\{|\Upsilon^*_Z J_0-J_0-\mathfrak{L}_{Z}J_0|(x)+ \\
& |\Upsilon^*_Z \phi-\phi-\mathfrak{L}_Z \phi|(x)\} + |\Upsilon^*_Z Q-Q-\mathfrak{L}_Z Q|(x)\}\\
&\leq C\Vert Z\Vert_{C^{k+1,\alpha}_{1}}^2,
\end{split}
\end{align}
where $C$ is uniform for all $Z\in \WW$ of sufficiently small norm. 
And similarly, for any $Z,Z'$ in a small neighborhood of $0$, 
\begin{align}
\Upsilon_Z^* J_1 - \Upsilon_{Z'}^* J_1 &= \Upsilon_{Z'}^*((\Upsilon^*_{Z'})^{-1}\circ\Upsilon^*_Z J_1 - J_1) \\
&= \Upsilon_{Z'}^*((Id - \mathfrak{L}_{Z'}+ R')\circ(Id+\mathfrak{L}_Z+ R) J_1 - J_1)
\end{align}
where $R,R'$ represents the corresponding higher order terms, and has estimate as in \eqref{estimateR}. Then we have
\begin{align}
\label{furtherwl3}
r^{3-\epsilon}|\Upsilon^*_ZJ_1-\Upsilon^*_{Z'}J_1-\mathfrak{L}_{Z-Z'}J_1|\leq C\cdot \|Z-Z'\|_{C^{k+1,\alpha}_1}
\cdot(\|Z\|_{C^{k+1,\alpha}_1}+ \|Z'\|_{C^{k+1,\alpha}_1})
\end{align}
Higher order derivative estimates are similar, and are omitted, and \eqref{wl3} follows.
Then by \eqref{expansionE} and \eqref{wl3} 
we can see that $\Upsilon^*$ is differentiable, and satisfies the condition of Lemma~\ref{l2.4}. 
A similar estimate of higher order expansions of $\Upsilon^*$ shows that it is a smooth map near $(0,0)$. 
\end{proof}
Recall that 
\begin{align}
\VV = \{ \theta \in \H_{-3}(X,End_a(TX)) \  | \ 
\theta = \mathfrak{L}_Z J, \  Z \in \WW \}.
\end{align}
The deformations of complex structure arising from elements of $\VV$
become trivial if one takes into account changes of 
coordinates at infinity, as the following theorem shows. 
\begin{theorem}
\label{t5.2}
Let $(X,J_0,g_0)$ be a K\"ahler ALE surface with $J_0 \in C^{\infty}$ and $g_0 \in C^{\infty}$. There exists an $\epsilon_1'>0$ such that for any complex structure $\Vert J_1-J_0\Vert_{C^{k,\alpha}_\delta}<\epsilon_1'$, where $k\geq 3, \alpha\in (0,1),\delta\in (-2,-1)$, there exists a
diffeomorphism $\Phi$,
of the form $\Phi_{Y_1} \circ \Upsilon_Z \circ \Phi_{Y_2}$ for 
$Y_1,Y_2 \in C^{k+1,\alpha}_{\delta+1}(TX)$
and $Z \in \WW$, such that $\nabla^*_{g_0}(E_{J_0}^{-1}(\Phi^*(J_1))) = 0$,
and such that $E_{J_0}^{-1}(\Phi^*(J_1))$ is   $L^2$-orthogonal to $\VV$.
Furthermore, there exists a constant $C$ so that 
\begin{align}
\label{closej}
\Vert E_{J_0}^{-1}(\Phi^* J_1)  \Vert_{C^{k,\alpha}_{-3}} < C \epsilon_1'.
\end{align}
\end{theorem}
\begin{proof}
Let $J_1 = J_0(\phi)$, where $\phi\in C^{k,\alpha}_{\delta}(X,\Lambda^{0,1}\otimes\Theta)$, and $\bar\partial\phi+[\phi,\phi]=0$. By Lemma~\ref{dfree} and Theorem 
\ref{kurthm}, without loss of generality, we may assume that 
$\phi = \phi(t) \in C^{k,\alpha}_{-3}(X, \Lambda^{0,1} \otimes \Theta)$, where $t\in B_{\epsilon_1}\subset \H_{-3}(X,\Lambda^{0,1}\otimes\Theta)$. 
In the following of the proof, we will fixed the ALE coordinate as in Lemma \ref{Zlem}, and again for simplicity,
we will assume that $Z_1\neq 0$.

Define the operator $\mathfrak{N}$ as
\begin{align}
\begin{split}
&C^{k+1,\alpha}_{-2+\epsilon}(TX)\times  
\mathcal{H}_{-3}(X, End_a(TX))
\times\WW \rightarrow C^{k-1,\alpha}_{-4+\epsilon}(TX)\times \RR^m\\
& (Y,Re(t),Z)\mapsto \\
&\Big\{ \nabla^*_{{\Phi_Y}_*g_0}E_{{\Phi_Y}_*J_0}^{-1}(\Upsilon_Z^*J_0(\phi(t))),  
\int \Big\langle E_{{\Phi_Y}_*J_0}^{-1} (\Upsilon_Z^*J_0(\phi(t))), {\Phi_Y}_*v_1
\Big\rangle_{{\Phi_Y}_*g_0} dV_{{\Phi_Y}_*g_0}, \\
& \hspace{5cm}\cdots,
\int \Big\langle E_{{\Phi_Y}_*J_0}^{-1} (\Upsilon_Z^*J_0(\phi(t))), {\Phi_Y}_*v_m \Big\rangle_{{\Phi_Y}_*g_0}dV_{{\Phi_Y}_*g_0}\Big\},
\end{split}
\end{align}
where $0<\epsilon\ll 1$, $\nabla^*_{{\Phi_Y}_*g_0}$ is the divergence operator associated to the metric 
${\Phi_Y}_*g_0$, $m = \dim(\VV)$, $\{v_1,\ldots,v_m\}$ is an orthonormal basis of $\VV$. 
Note that the method of (\ref{trick}) is used here again to make the operator smooth.
Also note that the map 
\begin{align}
\label{N1}
(Y, Re(t), Z) \mapsto \nabla^*_{{\Phi_Y}_*g_0}E_{{\Phi_Y}_*J_0}^{-1}(\Upsilon_Z^*J_0(\phi(t)))
\end{align}
is defined by choosing the first parameter of $\mathfrak{P}$
to be $(\Phi_Y)_*g_0-g_0$, 
choosing the second parameter to be $(\Phi_Y)_*J_0-J_0$,
choosing the third parameter to be $\phi_Z(t)$ (defined in \eqref{Upsilon}),
and letting $\delta' = -3+\epsilon$.
Since $\mathfrak{P}$ is smooth by Lemma \ref{smoothop}, and since the mappings
$Y \mapsto (\Phi_Y)_*g_0-g_0$, and $Y \mapsto (\Phi_Y)_*J_0-J_0$ 
are smooth by Lemma \ref{smooth}, and since 
$\Upsilon^*$ is smooth by Lemma \ref{W}, we have that \eqref{N1} is a smooth mapping.
By applying the same argument to the inner products in $\mathfrak{N}$, we can show that $\mathfrak{N}$ is smooth 
in an open neighborhood of $(0,0,0)$.

At $(0,0,0)$, the linearization of $\mathfrak{N}$ restricted to the tangent spaces of the first and third variables is:
\begin{align}
\begin{split}
& D\,\mathfrak{N}_0: C^{k+1,\alpha}_{-2+\epsilon}(TX)\times \WW \rightarrow C^{k-1,\alpha}_{-4+\epsilon}(TX)\times\RR^m \\
&(Y,Z) \mapsto \Big\{\nabla^*_{g_0}\mathfrak{L}_YJ_0 = \square_{g_0} Y, \int \Big\langle \frac{-1}{2}J_0\mathfrak{L}_{Z}J_0-\frac{1}{2}J_0\mathfrak{L}_Y J_0,v_1 \Big\rangle_{g_0} dV_{g_0},\\
& \hspace{5cm}\cdots,\int \Big\langle \frac{-1}{2}J_0 \mathfrak{L}_{Z}J_0-\frac{1}{2}J_0 \mathfrak{L}_Y J_0,v_m \Big\rangle_{g_0} dV_{g_0}\Big\}.
\end{split}
\end{align}
Restricting the domain and image, consider $D\mathfrak{N}_0$ mapping
from $C^{k+1,\alpha}_{-2+\epsilon}(TX)$ to $C^{k-1,\alpha}_{-4+\epsilon}(TX)$.
The adjoint of this restricted mapping, which we again denote by
$D\,\mathfrak{N}_0$, is:
\begin{align}
\begin{split}
(D\,\mathfrak{N}_0)^*: C^{k+1,\alpha}_{-\epsilon}(TX)&\rightarrow C^{k-1,\alpha}_{-2-\epsilon}(TX) \\
\eta &\mapsto (\square_{g_0} \eta)
\end{split}
\end{align}
The kernel of $(D\,\mathfrak{N}_0)^*$ consists of those $\eta\in C^{k+1,\alpha}_{-\epsilon}(TX)$ such that $(\square_{g_0} \eta)=0$, which implies that $\eta$ is the real part of a holomorphic vector field. 
From Proposition \ref{p2.5}, $\eta=0$. Since the cokernel of $D\,\mathfrak{N}_0$ is trivial, $D\,\mathfrak{N}_0$ is surjective on the restricted domain and range. 

It follows that the full mapping  $D \mathfrak{N}_0$ is surjective, 
since $\VV$ is generated by $\frac{-1}{2}J_0\mathfrak{L}_Z J_0$ for $Z\in \WW$.
Clearly then, in a small neighborhood of $(0,0,0)$,  $D \,\mathfrak{N}$ is surjective.
Since the mapping $\mathfrak{N}$ is smooth, Condition \eqref{Q} is satisfied,
so the existence of diffeomorphism $\Phi$ then follows from Lemma~\ref{l2.4}. 
Finally, the estimate \eqref{closej} follows from estimates above.
\end{proof}

\section{Stability of K\"ahler metrics}
\label{stable}
In this section, we want to generalize Kodaira-Spencer's stability theory of K\"ahler metrics to ALE surfaces to the family of complex structures found above \cite{KSIII}. Specifically, we want to find a smooth family of K\"ahler forms on $(X, J_t)$ for $t$ small, where $J_t$ for $t \in  B_{\epsilon_1}\subset \H_{ess}(X,\Lambda^{0,1}\otimes\Theta)$ is the family of complex structures constructed in Section~\ref{Kuranishi}.
In a slight abuse of notation, the notation $X_t$ will stand for $(X,J_t)$. 
All the norms below in this section are based on the central metric $g_0$.
\begin{lemma}The integer-valued function 
 $\dim(\H_{\tau}(X_t,\Lambda^{p,q}))$ is upper semicontinuous at $t=0$, 
for any $\tau\in (-2,0)$.
\end{lemma}
\begin{proof}
The Hodge Laplacian $\Delta_d = dd^*+d^*d$ is a Fredholm operator on $C^{k,\alpha}_{\tau}$ $(-2<\tau<0,k\geq 2)$. The following is proved in \cite{Bartnik}: If $P$ is a pseudo-differential operator such that $\Vert P-\frac{1}{2}\Delta_d\Vert _{op}<\epsilon$ for some $\epsilon$ sufficiently small, then $\ker(P)\leq \ker(\Delta_d)$, where $\Vert \cdot\Vert _{op}$ is the norm of operators from $C^{k,\alpha}_{\tau}$ to $C^{k-2,\alpha}_{\tau-2}$, which is defined as: 
\begin{align}
\Vert P\Vert _{op} = sup\{\Vert P(u)\Vert _{C^{k-2,\alpha}_{\tau-2}}: \Vert u\Vert _{C^{k,\alpha}_\tau}=1\}.
\end{align}
We may apply this to the operator $P = \square_t = \bar{\partial_t}^*\bar{\partial_t}+\bar{\partial_t}\bar{\partial_t}^*$, since we have $\Vert \square_t-\frac{1}{2}\Delta_d\Vert _{op}<\epsilon$ when $t\in \Delta^{d_1}_{\gamma}$ is small enough (when $t=0$, $\square_0 = \frac{1}{2}\Delta_d$ by the K\"ahler identities). Then
\begin{align}
\dim(\H_{\tau}(X_t,\Lambda^{p,q}))
\leq \dim(\H_{\tau}(X_0,\Lambda^{p,q})).
\end{align}
\end{proof}
Let $e^{\cdot, \cdot}_t$ denote the dimension of the space 
of $\square_t$-harmonic forms with decay rate of~$\tau$. 
Note that by a similar argument as in Proposition \ref{dcyprop}, 
any such harmonic form must decay like $O(r^{-3})$ as 
$r \rightarrow \infty$. 
So by \cite[Sections 8.4 and 8.9]{Joyce2000}, and Proposition \ref{p2.8}, we have
\begin{align}
b^{2,0}_0=b^{0,2}_0 = b^{1,0}_0=b^{0,1}_0=0.
\end{align}
From semi-continuity, it follows that 
\begin{align}
e^{2,0}_t=e^{0,2}_t=e^{1,0}_t=e^{0,1}_t=0.
\end{align}
The proof of the following theorem is inspired by \cite{BiquardRollin}.
\begin{theorem}
\label{orthogonal}
By choosing $\epsilon_1$ small enough, for the family of deformation $(X,J_t)$, 
$t \in B_{\epsilon_1}$, there exists a smooth family of K\"ahler forms ${\omega}_t$ on $X_t$, and 
\begin{align}
\label{ort}
\Vert {\omega}_t-\omega_0\Vert _{C^{k,\alpha}_\delta}\leq C\cdot \Vert t\Vert_{C^{k,\alpha}_{-3}} \leq C\cdot\epsilon_1
\end{align}
for some constant $C$, and some $k\geq 3, \delta\in (-2,-1)$. 

Furthermore, letting $\mathfrak{G}$ denote the group of holomorphic isometries of $(X,g_0,J_0)$, the mapping 
\begin{align}
\begin{split}
\aleph : B_{\epsilon_1} &\rightarrow \Lambda^{2}(T^*X)\\
& t \mapsto \omega_t
\end{split}
\end{align}
can be chosen equivariantly with respect to the action of $\mathfrak{G}$.
That is,  $ \aleph ( \iota^* t) = \iota^*  \aleph (t)$, for all 
$\iota \in \mathfrak{G}$.
\end{theorem}
\begin{proof}
Let $g_t = \frac{1}{2}(g+g\circ J_t)$ be the Hermitian metric on $X_t$, where $g$ is the Riemannian metric of the central fiber and $J_t$ is the complex structure of $X_t$.
Let $\square_t = \bar{\partial_t}^*\bar{\partial_t}+\bar{\partial_t}\bar{\partial_t}^*$ be the $\bar{\partial_t}-$Laplacian defined with respect to $g_t$. Let $\omega'_t = \frac{\sqrt{-1}}{2}g_t(J_t\,,\,)$ be the corresponding $(1,1)$-form. We want to perturb $\omega'_t$ to acquire a K\"ahler form ${\omega}_t$. 
We have
$\omega'_t = \omega^h_t+\phi_t$, where $\omega^h_t$ is the $\square_t$-harmonic part of $\omega'_t$, which is asymptotic to $dz_t\wedge {d{\bar z_t}}$ at the rate of $-\delta$, with $\Vert \phi_t\Vert _{C^{k,\alpha}_{\delta}}\leq C\Vert\square_t\omega'_t\Vert_{C^{k-2,\alpha}_{\delta-2}}$. 
We next show that this claimed decomposition is valid. 
Consider the following Fredholm operator
\begin{align}
\square_t: C^{k,\alpha}_{\delta}(X_t,\Lambda^{1,1})\rightarrow C^{k-2,\alpha}_{\delta-2}(X_t,\Lambda^{1,1}).
\end{align}
The kernel $\ker\square_t^*\subset C^{k,\alpha}_{-2-\delta}(X_t,\Lambda^{1,1})$ is of finite dimension. For any $\sigma\in \ker\square_t^*$, $\sigma$ admits an expansion 
\begin{align}
\sigma = \sum_{k=1}^2 \frac{a_k}{r^2}dz_t\wedge d\bar{z}_t + O(r^{-3})
\end{align}
as $r\to\infty$, where $a_k$ are constants. Since
\begin{align}
\begin{split}
0 = \int_X \langle \sigma, \square_t\sigma \rangle_{g_t} dV_{g_t} = \lim_{R\to\infty} &\Big\{ \int_{B(R)} \big(\langle \bar\partial\sigma, \bar\partial\sigma \rangle_{g_t} +\langle \bar\partial^*\sigma, \bar\partial^*\sigma \rangle_{g_t} \big) dV_{g_t} \\
&+ \int_{S(R)} \big( \langle \sigma, \bar\partial r \lrcorner \bar\partial\sigma \rangle_{g_t}  + \langle \bar\partial^*\sigma, \bar\partial r \lrcorner \sigma \rangle_{g_t} \big) dA_{g_t}  \Big\}\\
&= \int_X \big( \langle \bar\partial\sigma, \bar\partial\sigma \rangle_{g_t}   +\langle \bar\partial^*\sigma, \bar\partial^*\sigma \rangle_{g_t} \big) dV_{g_t} .
\end{split}
\end{align}
This shows that $\bar\partial\sigma = 0$, $\bar\partial^*\sigma=0$. Since also $\bar\partial_t-\bar\partial_{Euc} = O(r^{-3})$, we have $\bar\partial_{Euc}\sigma = O(r^{-5})$. Then $a_k=0$ and $\sigma = O(r^{-3})$. This implies that $\square_t\omega'_t\in (\ker\square_t^*)^{\perp}\cap C^{k-2,\alpha}_{\delta-2}(X_t,\Lambda^{1,1})$, which is the $L^2$-complement of $\ker\square_t^*$ in $C^{k-2,\alpha}_{\delta-2}(X_t,\Lambda^{1,1})$. Then there exists a $\phi_t\in C^{k,\alpha}_\delta(X_t,\Lambda^{1,1})$ such that $\square_t\phi_t = \square_t\omega'_t$ and $\Vert\phi_t\Vert_{C^{k,\alpha}_\delta}\leq C\Vert\square_t\omega'_t\Vert_{C^{k-2,\alpha}_{\delta-2}}$. Then we can let $\omega^h_t = \omega'_t-\phi_t$ which is $\square_t$-harmonic. 

\begin{proposition}If 
$\square_t\omega_t^h=0$, then $\bar{\partial}_t \omega^h_t = 0$.
\end{proposition}
\begin{proof}
$\bar{\partial}_t\square_t\omega^h_t = \square_t\bar{\partial}_t\omega_t^h=0$, then $\bar{\partial}_t\omega^h_t$ is $\square_t$-harmonic, with decay rate of $-(\delta-1)$. However, $e^{2,1}_t=e^{0,1}_t=0$ by the upper semi-continuity and since the conjugate Hodge-star operator maps harmonic forms to 
harmonic forms. Then $\bar{\partial_t}\omega^h_t=0$.
\end{proof}
We want to find ${\omega}_t$ such that $d{\omega}_t=0$; for this we need a lemma.
\begin{lemma} We have $\partial_t\omega^h_t = \bar{\partial}_t a_t$, where $a_t\in C^{k,\alpha}_{\delta}(X_t,\Lambda^{2,0})$, $k\geq 3$. Furthermore, 
\begin{align}
\Vert a_t\Vert _{C^{k,\alpha}_{\delta}}\leq C\Vert \partial_t\omega^h_t\Vert _{C^{k-1,\alpha}_{\delta-1}} \leq C'\Vert t\Vert_{C^{k,\alpha}_\delta}
\end{align}
\end{lemma}
\begin{proof}First note that
\begin{align} 
\bar{\partial}_t\partial_t\omega^h_t = -\partial_t\bar{\partial}_t\omega^h_t=0.
\end{align}
Then
\begin{align}
\partial_t\omega^h_t\in C^{k-1,\alpha}_{\delta-1}(X_t,\Lambda^{2,1}) = \H_{\delta-1}(X_t,\Lambda^{2,1})\oplus\square_t C^{k+1,\alpha}_{\delta+1}(X_t,\Lambda^{2,1}).
\end{align}
(This decomposition follows since $\square_t$ is Fredholm on $C^{k,\alpha}_{\delta+1}(X_t,\Lambda^{2,1})$.)
Then 
\begin{align}
\partial_t\omega^h_t = h_t+(\bar\partial_t\bar\partial_t^*+\bar\partial_t^*\bar\partial_t)f_t
\end{align}
with $f_t\in C^{k+1,\alpha}_{\delta+1}(X_t,\Lambda^{2,1})$, where $h_t$ is the harmonic part. 
Since $\bar\partial_t\partial_t\omega_t^H=0$, it follows that $ \square_t \bar\partial_t f_t=0$. 
From injectivity of $\square_t$ on $C^{k, \alpha}_{\delta}$, we have $ \bar\partial_t f_t=0$. Then 
\begin{align}
\partial_t\omega^h_t=h_t + \bar\partial_t\bar\partial_t^* f_t = h_t+\bar\partial_t a_t,
\end{align}
where $a_t = \bar\partial_t^* f_t\in C^{k,\alpha}_{\delta}(X_t,\Lambda^{2,0})$.
Since $e^{2,1}_t = 0$, then $h_t = 0$, and $\partial_t\omega^h_t = \bar\partial_t a_t$.\\
Since $\dim ( \mathcal{H}_{\delta+1}(X_t,\Lambda^{2,1}))=0$,
\begin{align}
\square_t: C^{k+1,\alpha}_{\delta+1}(X_t,\Lambda^{2,1})\rightarrow C^{k-1,\alpha}_{\delta-1}(X_t,\Lambda^{2,1})
\end{align} 
is an isomorphism. Then $\Vert f_t\Vert _{C^{k+1,\alpha}_{\delta+1}}\leq C\Vert \partial_t\omega^h_t\Vert _{C^{k-1,\alpha}_{\delta-1}}$ implies that
\begin{align} 
\Vert a\Vert _{C^{k,\alpha}_{\delta}}\leq C\Vert \partial_t\omega^h_t\Vert _{C^{k-1,\alpha}_{\delta-1}} \leq C'\Vert t\Vert_{C^{k,\alpha}_\delta},
\end{align}
\end{proof}
Now we have $a_t\in C^{k,\alpha}_{\delta}(X_t,\Lambda^{2,0})$ and $\bar\partial_t \bar{a_t} = 0$. Let $\delta+2$ not be a indicial root of $\square_t$, then 
\begin{align}
\square_t: C^{k+2,\alpha}_{\delta+2}(X_t,\Lambda^{0,2})\rightarrow C^{k,\alpha}_{\delta}(X_t,\Lambda^{0,2})
\end{align}
is Fredholm. Since $\dim(\mathcal{H}_{-4-\delta}(X_t,\Lambda^{0,2}))=0$, it follows that $\square_t$ is surjective. 
We can therefore choose $f_t\in C^{k+2,\alpha}_{\delta+2}(X_t,\Lambda^{0,2})$ such that $\bar{a_t} = \square_t f_t = \bar\partial_t\bar\partial_t^* f_t$ with 
\begin{align}
\Vert f_t\Vert _{C^{k+2,\alpha}_{\delta+2}}\leq C\Vert \bar{a_t}\Vert _{C^{k,\alpha}_{\delta}} \leq C'\Vert t \Vert_{C^{k,\alpha}_\delta}.
\end{align}
Let $\bar{\beta_t} = \bar\partial_t^* f_t$, and let ${\omega}_t = \omega^h_t+\bar\partial_t\beta_t$. Then $d{\omega}_t = \partial_t \omega^h_t-\bar\partial_t a_t = 0$. Choose $\epsilon_1$ to be small enough, then ${\omega}_t$ is a closed positive (1,1)-form.
\end{proof}

\section{Deformations of scalar-flat K\"ahler metrics}
\label{DSFK}
Notice that the previous sections did not use the scalar-flat assumption; this section is where we start using it. 
\subsection{The linearized operator}
\label{implicit}
In previous sections, we have shown that there exists a family of decaying deformation of complex structure with dimension equal to $\dim(\H_{ess}(X,\Lambda^{0,1}\otimes\Theta))$, and there exists a smooth family of K\"ahler forms along the deformation. In this section, we will show that near an initial scalar-flat K\"ahler ALE metric $\omega_0$, there is a smooth family of scalar-flat K\"ahler ALE metrics under these deformations. The arguments in this section are inspired by \cite{LS93, LeBrun_Simanca}. 

Denote $S(\omega_0+\sqrt{-1}\partial\bar\partial f)$ as the scalar curvature of $X$ with metric $\omega_0+\sqrt{-1}\partial\bar\partial f$. We now consider $S$ as mapping between weighted H\"older spaces,
with some $0<\epsilon\ll 1$
\begin{align}
\begin{split}
S : C^{k,\alpha}_{\epsilon}(X)&\rightarrow C^{k-4,\alpha}_{\epsilon-4}(X)\\
 f &\mapsto S(\omega_0+\sqrt{-1}\partial\bar\partial f).
\end{split}
\end{align}
Its linearization and adjoint are maps between:
\begin{align}
&L: C^{k,\alpha}_{\epsilon}({X})\rightarrow C^{k-4,\alpha}_{\epsilon-4}({X})\\
&L^*: C^{k,\alpha}_{-\epsilon}({X})\rightarrow C^{k-4,\alpha}_{-4-\epsilon}({X}).
\end{align}
By direct calculation, 
\begin{align}
L(f) &:= D\,S(\omega_0+\sqrt{-1}\partial\bar\partial f) = -(\Delta^2 f + R_{i,\bar{j}}\nabla^i\nabla^{\bar{j}}f). 
\end{align}
Define $\bar\partial^{\#}f = g_0^{i,\bar{j}}\bar\partial_j f$.  Also by direct calculation,\begin{align}
(\bar\partial\bar\partial^{\#})^*(\bar\partial\bar\partial^{\#})f &= \Delta^2 f +R_{i,\bar{j}}\nabla^i\nabla^{\bar{j}}f+\nabla_k S\nabla^k f.
\end{align}
Clearly then, when $\omega_0$ is scalar-flat, $L(f) = - (\bar\partial\bar\partial^{\#})^*(\bar\partial\bar\partial^{\#})(f)$, see \cite{LeBrun_Simanca}.

The following lemma shows that the linearized mapping is surjective. This can also be interpreted as the nondegeneracy of the Futaki invariant, see \cite{FutakiLNIM} and \cite[Section~3.2]{LeBrun_Simanca}.  
\begin{lemma}
For  $0<\epsilon\ll 1$, $k \geq 4$, $L$ is surjective.
\end{lemma}
\begin{proof}
Obviously, $L$ is an elliptic operator, and the indicial roots of $L$ are 
the integers. Consequently, by standard
weighted space theory, $L$ is Fredholm since $\delta$ is non-integral~\cite{LockhartMcOwen}. 

Let $f\in \ker(L^*)$, then $f = O(r^{-\epsilon})$ and $(\bar\partial\bar\partial^{\#})^*(\bar\partial\bar\partial^{\#})f=0$. Since 
\begin{align}
f(\bar\partial\bar\partial^{\#})^*(\bar\partial\bar\partial^{\#})f = o(r^{-4}),
\end{align}
integrating by parts, 
It follows that $\bar\partial\bar\partial^{\#} f = 0$, so $\bar\partial^{\#} f$ is a holomorphic vector field 
on $X$ which has a decay rate of $-(\epsilon-1)$, and Proposition \ref{p2.5} then 
implies that $\bar\partial^{\#} f = 0$. 
Since $f$ is a real-valued function, $f$ must be constant; but since $f=O(r^{-\epsilon})$, we have $f \equiv 0$. 
We have shown that $\ker(L^*) = \{0\}$, so $L$ is surjective. 
\end{proof}

\begin{proposition}
\label{lker}
For $0 < \epsilon\ll 1$ and $k \geq 4$, $\ker(L)$ is spanned by the constants. 
\end{proposition}
\begin{proof}
The operator $L$ admits an expansion near infinity of the form 
$L = \Delta_{Euc}^2 + Q$, where $Q$ represents lower order terms. 
Let
\begin{align}
N(L,\tau) = \mathrm{Index}(L: C^{k,\alpha}_\tau(X)\rightarrow C^{k-4,\alpha}_{\tau-4}(X)).
\end{align} 
Then, by the relative index theorem of Lockhart-McOwen,  
\begin{align}
N(L,\epsilon) - N(L, -\epsilon) =  N(\Delta_{Euc}^2,\epsilon) -  
N(\Delta_{Euc}^2,-\epsilon)  = 2.
\end{align}
see \cite{LockhartMcOwen}. The above argument shows that $Ker(L) \cap C^{k,\alpha}_{-\epsilon}(X) = \{0\}$, which implies that
\begin{align}
\ker(L) \cap C^{k,\alpha}_{\epsilon}(X) = \{ \text{constants} \}.
\end{align}
\end{proof}

\subsection{Construction of $\mathfrak{F}$}
\label{construct}
Now we start the construction of the class $\mathfrak{F}$. First, we show a weighted version $\partial\bar\partial-$lemma on ALE surfaces. This is similar 
to \cite[Theorem 8.4.4]{Joyce2000}, but with slightly different assumptions. 
\begin{lemma}
\label{l4.2}
Consider the K\"ahler ALE surface $(X,J_0,g_0)$. Let $\zeta\in C^{k,\alpha}_{\delta}(X,\Lambda^{1,1})$ be a real $(1,1)$ form, 
for some $k\geq 2$ and $\delta\in (-2,-1)$, satisfying $d\zeta=0$. 
Then there exists $\theta\in C^{k+2,\alpha}_{\delta+2}(X)$ such that 
$\zeta= \zeta_h+\sqrt{-1}\partial\bar\partial\theta$, and $\Vert \theta \Vert_{C^{k+2,\alpha}_{\delta+2}}\leq C\cdot \Vert \zeta-\zeta_h\Vert_{C^{k,\alpha}_\delta}$, where $\zeta_h\in \H_{\delta}(X,\Lambda^{1,1})$.
\end{lemma}
\begin{proof}
Since $\zeta$ is closed, can be decomposed as 
$\zeta= \zeta_h+\partial\eta+\bar\partial\xi$, where $\zeta_h\in \H_{\delta}(X,\Lambda^{1,1})$, $\eta\in C^{k+1,\alpha}_{\delta+1}(X,\Lambda^{0,1})$, and $\xi\in C^{k+1,\alpha}_{\delta+1}(X,\Lambda^{1,0})$. Since $\zeta$ is a real form, we can assume that $\eta = \overline{\xi}$. Consider the operator $\square$ and its Fredholm adjoint $\square^*$:
\begin{align}
\begin{split}
\square: C^{k+3,\alpha}_{\delta+3}(X,\Lambda^{0,1})&\rightarrow C^{k+1,\alpha}_{\delta+1}(X,\Lambda^{0,1})\\
\square^*: C^{k+3,\alpha}_{-3-\delta}(X,\Lambda^{0,1})&\rightarrow C^{k+1,\alpha}_{-\delta-5}(X,\Lambda^{0,1})
\end{split}
\end{align}
This implies that $\square$ has finite-dimensional cokernel. Then we have 
\begin{align}\eta = \eta_h+\square\gamma = \eta_h+\bar\partial\bar\partial^*\gamma+\bar\partial^*\bar\partial\gamma,\end{align}
where $\gamma\in C^{k+3,\alpha}_{\delta+3}(X,\Lambda^{0,1})$ and $\eta_h$ is the $\square$-harmonic part. Without loss of generality, assume $\eta_h=0$. Since $g_0$ is K\"ahler, $\partial\bar\partial^*(\bar\partial\gamma) = -\bar\partial\partial^*(\bar\partial\gamma) = 0$, so we can assume that $\bar\partial\gamma = 0$. Then $\zeta= \zeta_h+\sqrt{-1}\partial\bar\partial\theta'$, where $\theta' = 2 Im(\bar\partial^*\gamma)\in C^{k+2,\alpha}_{\delta+2}(X)$.

Now consider the Fredholm operator
\begin{align}
F = \sqrt{-1}\partial\bar\partial: C^{k+2,\alpha}_{\delta+2}(X)\rightarrow C^{k,\alpha}_\delta(X,\Lambda^{1,1})
\end{align}
Let $(\ker F^*)^{\perp}$ be the $L^2$-complement of $\ker F^*$ in $C^{k,\alpha}_\delta(X,\Lambda^{1,1})$. The argument above shows that $\ker F^* =\ker\square^*$. Then $\zeta-\zeta_h\in (\ker F^*)^{\perp}$. By the same argument as in the proof of Theorem \ref{orthogonal}, we can show that there exists a $\theta\in C^{k+2,\alpha}_{\delta+2}(X)$, such that $\sqrt{-1}\partial\bar\partial \theta = \zeta-\zeta_h$, and $\Vert \theta\Vert_{C^{k+2,\alpha}_{\delta+2}}\leq C\Vert \zeta-\zeta_h\Vert_{C^{k,\alpha}_\delta}$ for some constant $C$.
\end{proof}
\begin{theorem}
\label{t4.3}
Let $(X,J_0,g_0)$ be a scalar-flat K\"ahler ALE metric on a surface $X$. 
Then there is family $\mathfrak{F}$ of scalar-flat K\"ahler metrics near $g_0$, parametrized by $B$, that is, there is a differentiable mapping 
\begin{align}
F : B^1_{\epsilon_1} \times B^2_{\epsilon_2} \rightarrow \mathcal{M}(X), 
\end{align}
into the space of smooth Riemannian metrics $\mathcal{M}(X)$  
with $\mathfrak{F} = F(B^1_{\epsilon_1} \times B^2_{\epsilon_2})$.
Furthermore, letting $\mathfrak{G}$ denote the group of holomorphic isometries of $(X, J_0, g_0)$, the mapping $F$ is equivariant with respect to the action of $\mathfrak{G}$. 
That is, for $\iota \in \mathfrak{G}$,  and $(t, \rho) \in B^1_{\epsilon_1} \times B^2_{\epsilon_2}$, we have $F( \iota^* t, \iota^* \rho)
= \iota^* F(t, \rho)$.
\end{theorem}
\begin{proof}
Recall that  
\begin{align}
B^1_{\epsilon_1} = B_{\epsilon_1}(0)\subset \H_{ess}(X_0,End_a(TX_0));\;B^2_{\epsilon_2} = B_{\epsilon_2}(0) \subset \H_\delta(X_0,\Lambda^{1,1})\simeq H^{1,1}(X_0).
\end{align}
Let $t$ denote the parameter of $B^1_{\epsilon_1}$. By Section \ref{stable}, $\dim(\H_\delta(X_t,\Lambda^{p,q}))$ (using the metric determined by $\omega_t$) is upper semicontinuous, 
and
\begin{align}
\begin{split}
\dim(\H_\delta(X_0,\Lambda^{1,1})) &= \dim(H^{1,1}(X_0))\\ 
\dim(\H_{\delta}(X_0,\Lambda^{2,0})) &= \dim(\H_\delta(X_0,\Lambda^{0,2})) = 0.
\end{split}
\end{align} 
So for $\epsilon_1$  sufficiently small,
\begin{align}
\dim(\H_\delta(X_t,\Lambda^{2,0})) = \dim(\H_\delta(X_t,\Lambda^{0,2})) = 0,
\end{align}
and also
\begin{align}
\begin{split}
\dim(\H_\delta(X_t,\Lambda^{2,0}))\geq \dim(H^{2,0}(X_t))\\
\dim(\H_\delta(X_t,\Lambda^{0,2}))\geq \dim(H^{0,2}(X_t)).
\end{split}
\end{align}
This implies that  $\dim(H^{2}(X_t)) = \dim(H^{1,1}(X_t))$ which is a topological invariant. Then
\begin{align} 
\begin{split}
\dim (H^2(X_t))\leq  \dim(\H_\delta(X_t,\Lambda^{1,1}))\leq \dim(\H_\delta(X_0,\Lambda^{1,1}))=\\
= \dim(H^2(X_0)) = \dim(H^2(X_t)).
\end{split}
\end{align} 
This implies that $\dim(\H_\delta(X_t,\Lambda^{1,1})) = \dim(H^{2}(X_t))$ is constant,
so choose a smooth family of isomorphisms $\psi_t$ which map 
$\H_\delta(X_0,\Lambda^{1,1})$ to $\H_\delta(X_t,\Lambda^{1,1})$ for $t$ sufficiently small.
Note that $\mathfrak{G}$ acts on $\H_\delta(X_0,\Lambda^{1,1})$,
and from Theorem \ref{orthogonal}, $\mathfrak{G}$ also acts on 
 $\H_\delta(X_t,\Lambda^{1,1})$. Clearly we can choose $\psi_t$ to 
be equivariant with respect to these actions. 

It has been shown in Section~\ref{stable} 
that there is a smooth family of K\"ahler forms $\omega_t$. Let $\rho\in B^2_{\epsilon_2}$, define 
\begin{align}
\label{omega}
\omega(t,\rho) = \omega_t+\psi_t(\rho)
\end{align}
Now consider the mapping
\begin{align}
\mathcal{S}: B^1_{\epsilon_1}\times B^2_{\epsilon_2}\times C^{k,\alpha}_{\epsilon}(X)\rightarrow C^{k-4,\alpha}_{\epsilon-4}(X),
\end{align}
defined by 
\begin{align}
\mathcal{S}: (t,\rho,f)\mapsto S(\omega(t,\rho)+\sqrt{-1}\partial_t\bar\partial_t f).
\end{align}
We endow the domain with the product norm, where $B^1_{\epsilon_1}$ and $B^2_{\epsilon_2}$ have 
the $L^2$-norm. By direct calculation, the linearization of $\mathcal{S}$ at $0$ is: 
\begin{align}
D \mathcal{S} = \begin{pmatrix} *, & -Ric_h, & L \end{pmatrix}
\end{align}
where $Ric_h$ is the harmonic part of Ricci form. Since $L$ is surjective as shown in Lemma \ref{l4.2}, $D\mathcal{S}$ is surjective by the lemma above. 
Next, we recall the expansion of the curvature tensor
\begin{align}
\label{expform}
Rm(g_0+h) = Rm(g_0)+(g_0+h)^{-1}*\nabla^2 h+ (g_0+h)^{-2}*\nabla h *\nabla h,
\end{align}
where $*$ denotes a various tensor contractions, and $h$ is a symmetric tensor
such that $g_0+h$ is a Riemannian metric \cite[Formula 3.40]{GurskyViaclovsky16}. 
In our case, $h$ can be written as 
\begin{align}
h( , ) = (\omega(t,\rho)+\sqrt{-1}\partial_t\bar\partial_t f)(-J_t,) -\omega(0,0)(-J_0,).
\end{align}
Next, by \eqref{ort}, we have the estimates
\begin{align}
\begin{split}
|h| &\leq C(|t|+|\rho|+|\nabla^2 f|) \leq C (r^{\delta}\Vert t\Vert_{C^{k,\alpha}_\delta}+r^{\delta}\Vert\rho\Vert_{C^{k,\alpha}_\delta}+r^{-2+\epsilon} \Vert f\Vert_{C^{k,\alpha}_\epsilon}) \\
|\nabla h| &\leq C(|\nabla t|+|\nabla \rho|+|\nabla^3 f|)\leq C (r^{-1+\delta}\Vert t\Vert_{C^{k,\alpha}_\delta}+r^{-1+\delta}\Vert\rho\Vert_{C^{k,\alpha}_\delta}+r^{-3+\epsilon}\Vert f\Vert_{C^{k,\alpha}_\epsilon}) \\
|\nabla^2 h| &\leq C(|\nabla^2 t|+|\nabla^2 \rho|+|\nabla^4 f|)\leq C (r^{-2+\delta}\Vert t\Vert_{C^{k,\alpha}_\delta}+r^{-2+\delta}\Vert\rho\Vert_{C^{k,\alpha}_\delta}+r^{-4+\epsilon}\Vert f\Vert_{C^{k,\alpha}_\epsilon}).
\end{split}
\end{align}
Then by \eqref{expform}, we have
\begin{align}
\begin{split}
S(g_0+h)-DS\cdot h = C_0&(g_0,h) * Rm(g_0)*h*h \\
&+ C_1(g_0,h)*h*\nabla^2 h + C_2(g_0,h)*\nabla h *\nabla h,
\end{split}
\end{align}
where $C_0(g_0,h), C_1(g_0,h),$ and $C_2(g_0,h)$ are bounded functions, when $h$ is sufficiently small. 
Without loss of generality, we can assume that $-2+\epsilon<\delta$. Then
\begin{align}
\begin{split}
|S(g_0+h)-DS\cdot h| &\leq C\{ r^{-2}|h|^2 + |h|\cdot|\nabla^2 h| + |\nabla h| \cdot |\nabla h|\}\\
&\leq C \cdot r^{-2+2\delta}\cdot (\Vert t\Vert_{C^{k,\alpha}_\delta}+\Vert \rho\Vert_{C^{k,\alpha}_\delta}+\Vert f\Vert_{C^{k,\alpha}_\epsilon})^2
\end{split}
\end{align}
which implies that 
\begin{align}
\Vert S(g_0+h)-DS\cdot h\Vert_{C^{k,\alpha}_{-4+\epsilon}} \leq C (\Vert t\Vert_{C^{k,\alpha}_\delta}+\Vert \rho\Vert_{C^{k,\alpha}_\delta}+\Vert f\Vert_{C^{k,\alpha}_\epsilon})^2.
\end{align}
Furthermore,
\begin{align}
\begin{split}
\Vert S(g_0+h)&-S(g_0+h')-DS\cdot (h-h')\Vert_{C^{k,\alpha}_{-4+\epsilon}} \leq \\
C &\cdot \big(\Vert t-t'\Vert_{C^{k,\alpha}_\delta}+\Vert \rho-\rho'\Vert_{C^{k,\alpha}_\delta}+\Vert f-f'\Vert_{C^{k,\alpha}_\epsilon}\big)\\
&\cdot \big(\Vert t\Vert_{C^{k,\alpha}_\delta}+\Vert \rho\Vert_{C^{k,\alpha}_\delta}+\Vert f\Vert_{C^{k,\alpha}_1}+\Vert t'\Vert_{C^{k,\alpha}_\delta}+\Vert \rho'\Vert_{C^{k,\alpha}_\delta}+\Vert f'\Vert_{C^{k,\alpha}_\epsilon} \big).
\end{split}
\end{align}
Since $B^1_{\epsilon_1}$ and $B^2_{\epsilon_2}$ are finite-dimensional, we can replace the corresponding 
norms by the $L^2$-norm. This shows that $\mathcal{S}$ satisfies the condition \eqref{Q}.
Then by Lemma \ref{l2.4}, the zero 
set of $\mathcal{S}$ is in one-to-one correspondence with the kernel of $\mathcal{S}'(0)$.
Since any solution of the nonlinear equation can be written 
as a kernel element plus a unique element in the image of a 
bounded right inverse for the linearized operator, 
the zero set of $\mathcal{S}$ can also be written as a graph over
$B^1_{\epsilon_1} \times B^2_{\epsilon_2}$. That is, for any $(t, \rho) \in B^1_{\epsilon_1} \times B^2_{\epsilon_2}$,
there exists a unique $f(t, \rho)$, up to constants, 
of
\begin{align}
\mathcal{S}(t, \rho, f(t,\rho) ) = 0,
\end{align} 
so a family of scalar-flat K\"ahler metrics over $B^1_{\epsilon_1} \times B^2_{\epsilon_2}$ can be constructed as 
\begin{align}
F: (t, \rho) \mapsto  g(t,\rho) = \Big( \omega(t,\rho) + \sqrt{-1} \partial_t \bar \partial_t f(t, \rho)\Big) (-J_t \cdot, \cdot),
\end{align}
where $g(0,0) = g_0$. It is a straightforward consequence of the implicit function theorem that the mapping $F$ is differentiable.
Then the image of $F$ gives us the family of scalar-flat K\"ahler metrics $\mathfrak{F}$, and the construction is clearly equivariant with respect to the action of~$\mathfrak{G}$. 

To finish the proof, we show that metrics in $\mathfrak{F}$ are smooth. 
By the result of \cite{TianViaclovsky}, any K\"ahler constant scalar 
curvature metric satisfies an equation of the form
\begin{align}
\Delta Ric = Rm * Ric,
\end{align}
where the right hand side denotes quadratic curvature contractions involving 
the full curvature tensor $Rm$ and the Ricci tensor. 
By a Moser iteration method and regularity bootstrap argument in harmonic coordinates,  it follows that $g$ is smooth, see \cite[Theorem 6.4]{TV2}.
\end{proof}

\section{Versality and uniqueness of the moduli space}
\label{vp}
In this section, all norms are defined based on the initial K\"ahler metric $g_0$.
We let $(g_1,J_1)$ be any scalar-flat K\"ahler metric satisfying $\Vert g_1-g_0\Vert _{C^{k,\alpha}_\delta}<\epsilon_3$, (for some $k\geq 4$, $\delta\in (-2,-1)$, and $\epsilon_3$ will be determined later). 
We begin with the following lemma. 
\begin{lemma}
\label{l6.1}
For $\epsilon_3$ sufficiently small, if $\Vert g_1-g_0\Vert _{C^{k,\alpha}_{\delta}}<\epsilon_3$, then there exists a diffeomorphism 
$\Phi_0: X \rightarrow X$ and constants $C_1, C_2$ depending upon $g_0, J_0$ so that 
\begin{align}
\label{goes}
\Vert \Phi^*_0 g_1-g_0\Vert _{C^{k,\alpha}_{\delta}}&<C_1 \epsilon_3\\
\label{j1es}
\Vert \Phi^*_0 J_1-J_0\Vert _{C^{k,\alpha}_\delta}&<C_2 \epsilon_3.
\end{align}
\end{lemma}
\begin{proof}
Let $\nabla_0$ denote the covariant derivative of $g_0$, and $\Gamma_0, \Gamma_1$ denote the Christoffel symbols of $g_0, g_1$, respectively. Then $\Vert \Gamma_1-\Gamma_0\Vert _{C^{k-1,\alpha}_{\delta-1}}<C\epsilon_3$, so in particular $|\Gamma_1-\Gamma_0|< \frac{C \epsilon_3}{(1+r)^{-\delta+1}}$ as $r \rightarrow \infty$.

Since $(g_i,J_i)$ is K\"ahler, $\nabla_i J_i = 0$ for $i = 0 ,1$. We now 
estimate $|J_1-J_0|(p)$ along any geodesic ray $\gamma$ starting at $p_0$.
From \cite[Lemma 1.1]{HL15}, both $J_1$ and $J_0$ have a finite limit
along $\gamma$ as $t \rightarrow \infty$, and furthermore
\begin{align}
J_i &= (J_i)_0 + O(r^{-2}),
\end{align}
where $(J_i)_0$ is a constant complex structure on $\RR^4$ for $i = 0,1$. 
We are assuming that $(J_0)_0 = J_{Euc}$. 
Clearly, there exists a diffeomorphism $\Phi_0 : X \rightarrow X$ so that 
\begin{align}
\Phi^*_0J_1 = J_{Euc} + O(r^{-2}), 
\end{align}
as $r \rightarrow \infty$ such that \eqref{goes} is satisfied. 
This can be done, for example, by connecting a rotation defined on the 
complement of a large ball $X \setminus B(p_0, 2R)$ 
to the identity transformation on $B(p_0,R)$ by smooth path 
of rotations of each sphere on the annulus $A(R,2R)$. 
We then have
\begin{align}
\begin{split}
|\Phi^*_0 J_1-J_0|(p) &= |\Phi^*_0 J_1-J_0|(p) - \lim_{t \rightarrow \infty}|
\Phi^*_0 J_1 - J_0|(\gamma(t))\\ 
& \leq \int_{s=r}^\infty |\nabla_0(\Phi^*_0 J_1-J_0)(\gamma(s)) |ds\\
&\leq \int_{s=r}^\infty |\nabla_0 \Phi^*_0 J_1(\gamma(s))|ds\\
&<\int_{s=r}^\infty\frac{C\cdot \epsilon_3}{(1+s)^{-\delta+1}}ds\leq \frac{C'\cdot\epsilon_3}{(1+r)^{-\delta}}.
\end{split}
\end{align}
Since this estimate is true along any ray, it 
follows that $\Vert \Phi^*_0 J_1-J_0\Vert _{C^{0}_\delta}<C' \epsilon_3$. 
We can estimate the higher regularities in the same way, and \eqref{j1es} follows.
\end{proof}
We next prove the ``versality'' of our family $\mathfrak{F}$. That is, for $g_1$ as above, there exists a diffeomorphism $\Phi$ such that $\Phi^*g_1$ is in the class $\mathfrak{F}$.
\begin{theorem}
\label{gthm}
 Let $-2 < \delta < -1$ be fixed.  
There exists an $\epsilon_3>0$ such that for any scalar-flat K\"ahler metric $(g_1,J_1)$ satisfying $\Vert g_1-g_0\Vert _{C^{k,\alpha}_\delta}\leq\epsilon_3$, there exists a diffeomorphism $\Phi: X \rightarrow X$  of the form $\Phi_0 \circ \Phi_{Y_1} \circ \Upsilon_Z \circ \Phi_{Y_2}$ where $\Phi_0$ is as in Lemma~\ref{l6.1}, $Y_1,Y_2 \in C^{k+1,\alpha}_{\delta+1}(TX)$, and $Z \in \WW$,
such that $\Phi^*g_1 \in \mathfrak{F}$.
Furthermore, there exists a constant $C$ so that
\begin{align}
\label{metest}
\Vert \Phi^*g_1-g_0\Vert_{C^{k,\alpha}_\delta}\leq C\epsilon_3.
\end{align}
\end{theorem}
\begin{proof}
Let $\Phi_0$ denote the diffeomorphism from Lemma \ref{l6.1}.
Then $\Phi_0^* J_1$ satisfies the assumptions of Theorem \ref{t5.2}, so 
there exists a diffeomorphism $\tilde{\Phi}$ satisyfing the properties stated in that theorem.
In particular, $\tilde{\Phi}$ is of the form $\tilde{\Phi} = \Phi_{Y_1} \circ \Upsilon_Z
\circ \Phi_{Y_2}$, where $Y_1,Y_2 \in C^{k+1,\alpha}_{\delta+1}(TX)$,
and $Z \in \WW$. The estimate \eqref{metest} for $\Phi_0 \circ \tilde{\Phi}$ 
is then proved as follows.
For $\Phi_{Y_1}$, we estimate 
\begin{align}
\label{phes}
\Vert \Phi^*_{Y_1} \Phi_0^* g_1 - g_0 \Vert_{C^{k,\alpha}_\delta}
\leq \Vert \Phi^*_{Y_1} \Phi_0^* g_1 - \Phi_0^* g_1 \Vert_{C^{k,\alpha}_\delta} + \Vert \Phi_0^* g_1 - g_0 \Vert_{C^{k,\alpha}_\delta}
\leq  C \epsilon_3.
\end{align}
This estimate holds since $\Phi^*_0 g_1$ is also a scalar-flat 
K\"ahler metric which is smooth by the last
observation in the proof of Theorem \ref{t4.3},
so when $\epsilon_3$ is sufficiently small, 
Lemma~\ref{smooth} holds for $g_1$.
The same argument applies to $\Phi_{Y_2}$. 
The estimate~\eqref{metest} then follows from \eqref{phes}, \eqref{wl2} and Lemma \ref{l6.1}.

Let $\omega_1$ denote the K\"ahler form of $\Phi^*g_1$.
It was shown in Section \ref{DSFK} that $\dim(H^{1,1}(X_t))$ is locally constant, which implies $H^{1,1}(X_0)\simeq H^{1,1}(X_t)$ for any $t\in B^1$. Then there exists a $\rho\in \H_{\delta}(X_0,\Lambda^{1,1})$ such that $[\omega(\phi_{ess},\rho)] = [\omega_1]\in H^{1,1}(X_t)$, where $\omega(\phi_{ess},\rho)$ is defined in \eqref{omega}. 
Let $g' = F(\phi_{ess},\rho)$, and let $\omega'$ denote the corresponding K\"ahler form. 
Then by Lemma \ref{l4.2}, $\omega'-\omega_1 = \sqrt{-1}\bar\partial_t\bar\partial_t^* f$ for some potential function $f\in C^{k+2,\alpha}_{\delta+2}(X)$, and $\Vert f\Vert_{C^{k+2,\alpha}_{\delta+2}}\leq C\Vert \omega'-\omega_1 \Vert_{C^{k,\alpha}_{\delta}}$.

As shown in the proof of Theorem \ref{t4.3}, in a neighborhood of $g_0$, a scalar-flat K\"ahler metric is uniquely determined by the K\"ahler class and complex structure. This implies that $g' = g_1$. 
\end{proof}
This completes the proof of Theorem \ref{tmain}. Next, we will complete the 
proof of Theorem \ref{t2}. 
From Theorem \ref{t4.3}, the mapping $F$ is $\mathfrak{G}$-equivariant. 
That is, for $\iota \in \mathfrak{G}$,  and $(t, \rho) \in B^1_{\epsilon_1} \times B^2_{\epsilon_2}$, we have $F( \iota^* t, \iota^* \rho)
= \iota^* F(t, \rho)$.
This clearly implies that two elements in the same 
orbit of $\mathfrak{G}$ are isometric, which proves the 
first statement in Theorem \ref{t2}.

For the second statement in Theorem \ref{t2}, we need the following lemma, 
which says that in the non-hyperk\"ahler case, the dimension of $\mathfrak{F}$ 
is the same as the dimension of the parameter space. 
\begin{lemma}
\label{non-hyperKahler}
Let $(X,g,J)$ be as above. If $g$ is not hyperk\"ahler then
$F$ is injective, so $\dim(\mathfrak{F}) = d = \dim(B^1_{\epsilon_1} \times B^2_{\epsilon_2})$.  
\end{lemma}
\begin{proof}
If $F(\phi_1, \rho_1) = F(\phi_2, \rho_2)  = \tilde{g}$, then 
the metric $\tilde{g}$ is K\"ahler with respect to the two 
complex structures $J_1$ and $J_2$, corresponding to 
$\phi_1$ and $\phi_2$, respectively. 

We have $J_i\in End(TX)$ ($i=1,2$) such that $J_i^2 = -1$. Locally, $J_i$ can be considered as an purely imaginary Hamiltonian number. Define 
\begin{align}
J_3 = \frac{J_1J_2-J_2J_1}{\Vert J_1J_2-J_2J_1\Vert },
\end{align}
where for $p \in X$, 
\begin{align}
\Vert J_i(p)\Vert = \sup_{0\neq v\in T_p X}\Big\{\frac{\Vert J_i(v)\Vert_h}{\Vert v\Vert_h}\Big\}.
\end{align} 
Then $J_3\in End(TX)$. We claim that $J_3$ is well-defined, $J_3^2=-1$, $J_1,J_2,J_3$ are linearly independent.  Clearly, 
\begin{align}
(J_1J_2-J_2J_1)(J_1J_2-J_2J_1) = -2 + (J_1J_2J_1J_2- J_2J_1J_2J_1).
\end{align} 
Since $J_1J_2J_1J_2-J_2J_1J_2J_1$ is real and has norm which is $<2$, 
we have that 
\begin{align}
(J_1J_2-J_2J_1)(J_1J_2-J_2J_1)\neq 0,
\end{align}
so $J_3$ is well-defined, and $J_3\cdot J_3=-1$. If we write $J_1,J_2$ locally as Hamiltonian numbers $a_1 I+b_1 J + c_1 K$, $a_2 I+b_2 J + c_2 K$ (where $a_i,b_i, c_i\in \RR$, $I^2 = J^2 = K^2 =-1$), then 
\begin{align}
J_1J_2-J_2J_1 = (b_0c_1-b_1c_0)I+(c_0a_1-c_1a_0)J+(a_0c_1-a_1c_0)K,
\end{align}
 which is linearly independent with $J_1$ and $J_2$.
Then $J_1,J_2,J_3$ are linearly independent. As a result,
$\{J_1,J_2,J_3\}$ gives a hyperk\"ahler structure. Then we have proved that, 
if $\Gamma\not\subset {\rm{SU}}(2)$, then $J_1 = J_2$, 
then by the proof of Theorem \ref{t4.3}, $F$ is injective.
\end{proof}
The next proposition immediately implies the second part of Theorem \ref{t2}. 
\begin{proposition}
For any two elements $g_1,g_2 \in \mathfrak{F}$ associated to complex structures $J_1,J_2$ such that 
$E_{J_0}^{-1}(J_1),E_{J_0}^{-1}(J_2)$ are divergence-free, 
if there exists a small diffeomorphism $\Phi_Y$ which is induced by the exponential map of a vector field $Y\in C^{k+1,\alpha}_{\delta+1}(TX)$, such that $g_1 = \Phi_Y^*g_2$, then $g_1,g_2$ are the same. 
\end{proposition}
\begin{proof}
First, assume $X$ is non-hyperk\"ahler. Since $\Phi_Y^*g_2 = g_1$, by Lemma \ref{non-hyperKahler}, $J_1=\Phi_Y^*J_2$. However by Lemma \ref{dfree}, there is a unique small diffeomorphism that gauges $E_{J_0}^{-1}(J_2)$ to be divergence-free. Then $\Phi_Y = Id$, $g_1 = g_2$.

When $X$ is hyperk\"ahler, then by rotating the hyperk\"ahler sphere, there exists a $J_1'$ such that $J_1'$ is compatible with $g_1$ and $(X,\Phi_Y^*J_2)$ is biholomorphic to $(X,J_1')$. Then by the same argument above, $g_1=g_2$.
\end{proof}

Since $\mathfrak{F}$ is of finite dimension, and $\mathfrak{G}$ is a compact group action on $\mathfrak{F}$, the dimension~$m$ of $\mathfrak{M} = \mathfrak{F}/\mathfrak{G}$ is well-defined. In the non-hyperk\"ahler case, by Lemma \ref{non-hyperKahler}, 
\begin{align}
m = d- ( \text{the dimension of a maximal orbit of } \mathfrak{G}).
\end{align}
(For the hyperk\"ahler case, recall Remark~\ref{hrem}.) 
\section{Deformations of the minimal resolution}
\label{defminsec}
In this section, we prove Theorems \ref{generalthm} and \ref{t1.8}.
\subsection{Harmonic representation of $H^1(X,\Theta)$}
\label{decayharmonic}
Let $X$ denote the minimal resolution of $\CC^2/\Gamma$, where $\Gamma$ is a finite subgroup of ${\rm{U}}(2)$ without complex reflections. In the following, we want to construct a weighted version of Hodge theory, that links the sheaf cohomology with the decaying harmonic forms. 

The divisor $E = \cup_i E_i$ is a union of irreducible components which are
rational curves, with only normal crossing singularities.
Let $Der_{E}({X})$ denote the sheaf dual to logarithmic $1$-forms along $E$ (see \cite{Kawamata}). 
We note that $Der_{E}({X})$ is a locally free sheaf of rank $2$, see \cite{Wahl1975}.
Away from $E$, this is clear. If $p\in E_i$, we can choose a holomorphic coordinate chart $\{z_1, z_2\}$ such that near $p$, $E_i = \{z_1=0\}$. Then local sections 
of $Der_{E}({X})$ are generated by $\{z_1\frac{\partial}{\partial z_1},\frac{\partial}{\partial z_2}\}$.

The short exact sequence
\begin{align}
0\rightarrow Der_E({X}) \rightarrow \Theta_{{X}} \rightarrow \oo_E(E)\rightarrow 0,
\end{align}
induces an exact sequence of cohomologies
\begin{align}
\label{es}
0\rightarrow H^1(X,Der_E(X))\rightarrow H^1(X,\Theta)\rightarrow H^1(E,\oo_E(E))\rightarrow H^2(X,Der_E(X)).
\end{align}
Since $E$ is composed of rational curves whose self-intersection numbers are negative, we have $H^0(E,\oo_E(E))=0$.
\begin{proposition}
We have the vanishing result: $H^1(X,Der_E(X))=0$.
\end{proposition}
\begin{proof}
This is a result which follows from work of 
\cite{BKR, Brieskorn, Laufer1973, Wahl1975}. 
\end{proof}

By Siu's vanishing theorem (\cite{Siu}), since $X$ is a non-compact $\sigma$-compact complex manifold, for any coherent analytic sheaf $\mathscr{F}$ on $X$, the top degree sheaf cohomology $H^2(X,\mathscr{F})$ is trivial. Consequently, $H^2(X,Der_E(X))=0$, which gives us an isomorphism $H^1(\Theta_{{X}}) = H^1(O_E(E))$. Let $-e_j$ be the self-intersection number of each rational curve $E_j\subset E$, and let $k_\Gamma$ be the number of rational curves (which is equal to $b_2$). Then 
\begin{align}
\dim(H^1(X,\Theta)) = \sum_{j=1}^{k_\Gamma}(e_j-1).
\end{align}

\begin{theorem}
\label{t2.3}
Let $(X,g,J)$ denote the minimal resolution of $\CC^2/\Gamma$ with any ALE K\"ahler metric $g$ of order $\tau > 1$. Then
\begin{align}
H^1(X,\Theta) \cong  \H_{-3}(X,\Lambda^{0,1}\otimes\Theta)
\cong \H_{ess}(X,\Lambda^{0,1}\otimes\Theta)
\end{align}
\end{theorem}
\begin{proof}
First, we consider the case on $\CC^2\setminus\{0\}$, and we compute $H^1(\CC^2\setminus \{0\},\Theta)$. The domain $\CC^2\setminus \{0\}$ can be covered by two charts: $U_1 = \{z_1\neq 0\} $ and $U_2 = \{z_2\neq 0\}$. Note that $U_1,U_2$ are each isomorphic to $\CC\times\CC^*$, and $U_1\cap U_2$ is isomorphic to $\CC^*\times\CC^*$. 
Then $\theta_1\in \Gamma(U_1, \Theta)$ can be expanded into a Laurent series
\begin{align}
\theta_1 = \sum_{i\in\ZZ_{\geq 0},j\in \ZZ}a^1_{i,j}z_1^iz_2^j\frac{\partial}{\partial z_1}+b^1_{i,j}z_1^iz_2^j\frac{\partial}{\partial z_2},
\end{align}
$\theta_2\in \Gamma(U_2, \Theta)$ can be expanded into 
\begin{align}
\theta_2 = \sum_{i\in\ZZ,j\in \ZZ_{\geq 0}}a^2_{i,j}z_1^iz_2^j\frac{\partial}{\partial z_1}+b^2_{i,j}z_1^iz_2^j\frac{\partial}{\partial z_2},
\end{align}
$\theta_{1,2}\in \Gamma(U_1\cap U_2, \Theta)$ can be expanded into 
\begin{align}
\theta_{1,2} = \sum_{i\in\ZZ,j\in \ZZ}a^{1,2}_{i,j}z_1^iz_2^j\frac{\partial}{\partial z_1}+b^{1,2}_{i,j}z_1^iz_2^j\frac{\partial}{\partial z_2}.
\end{align}
Then any $\theta_{1,2}\in \Gamma(U_1\cap U_2, \Theta)$ which cannot be represented as $\theta_1-\theta_2$ must be of the form
\begin{align}
\label{cce}
\theta_{1,2} = \sum_{i\in\ZZ_-,j\in \ZZ_-}a^{1,2}_{i,j}z_1^iz_2^j\frac{\partial}{\partial z_1}+b^{1,2}_{i,j}z_1^iz_2^j\frac{\partial}{\partial z_2}.
\end{align}

By Dolbeault's lemma, we have $H^1(X,\Theta)\simeq H^0(X,\Lambda^{0,1}\otimes\Theta)$ (\cite{GH}). Using this, we transform the \v{C}ech cohomology element \eqref{cce} 
to an element of Dolbeault cohomology $\phi\in \Gamma(\CC^2\setminus \{0\}, \Lambda^{0,1}\otimes\Theta)$.
Let $\rho_1,\rho_2$ be a partition of unity, where $\rho_1$ is supported in $U_1$ and $\rho_1 =1$ for $|z_1|>1$; $\rho_2$ is supported in $U_2$ and $\rho_2 = 1$ for $|z_2|>1$; $\rho_1+\rho_2=1$. Then define 
 $\phi\in \Gamma(\CC^2\setminus \{0\}, \Lambda^{0,1}\otimes\Theta)$ as
\begin{align}
\phi = \begin{cases}\bar\partial(\rho_2\cdot\theta_{1,2}) &\text{ on } U_1 \\ -\bar\partial(\rho_1\cdot\theta_{1,2}) &\text{ on } U_2\end{cases}.
\end{align}
The form $\phi$ is well-defined since $\bar\partial(\rho_2\cdot\theta_{1,2})+\bar\partial(\rho_1\cdot\theta_{1,2}) = \bar\partial(\theta_{1,2}) = 0$ on $U_1\cap U_2$. Furthermore, $\phi$ is decaying at infinity, since the degrees 
appearing in \eqref{cce} are all negative.

Now for any closed form $\phi\in H^0(X,\Lambda^{0,1}\otimes\Theta)$, since $X-E$ is biholomorphic to $\CC^2\setminus \{0\}/\Gamma$, $\phi$ corresponds with a $\Gamma$-equivariant form $\phi'\in H^0(\CC^2\setminus\{0\},\Lambda^{0,1}\otimes\Theta)$. Then $\phi' = \bar\partial\tau+\psi$, with $\psi$ a closed form and decaying at infinity by the argument above. Since $\Gamma$ is a finite group, by averaging $\tau$ by the group action, we can assume that $\tau$ is $\Gamma$-invariant. Let $\rho$ be a cutoff function on $X$ which equals to $0$ in $B(R)$ and $1$ outside of $B(2R)$. Then $\phi-\bar\partial(\rho\cdot\tau)$ is a closed decaying form, which is in the same class as $\phi$. 

This shows that the natural mapping 
\begin{align}
\label{natmap}
\H_{-3}(X,\Lambda^{0,1}\otimes\Theta) \rightarrow H^1(X,\Theta)
\end{align}
is surjective. We next show that this mapping is injective.

Let $\phi = \bar\partial \eta\in \H_{-3}(X,\Lambda^{0,1}\otimes\Theta)$, where $\eta\in \Gamma(X,\Theta)$. Then $[\phi] = 0\in H^1(X,\Theta)$. We want to prove $\phi = 0$ by showing that $\phi = \bar\partial\xi$ for some $\xi = O(r^{-2})$.
Let $\chi$ be a smooth cutoff function defined on $X$ with compact support that contains $E$. Then $\phi = \bar\partial(\chi \eta)+\bar\partial((1-\chi)\eta)$. Since $\pi$ is biholomorphic on $X\setminus E$, $(1-\chi)\eta$ can be pushed down by $\pi$ and extended to ${\zeta}$ on $\CC^2$ such that ${\zeta}=0$ in a neighborhood of the origin.
Also $\bar\partial\zeta = O(r^{-3})$ on $\CC^2$ by the decaying rate of $\phi$.
Exactly as in \eqref{Poin} above, by the Poincar\'e lemma, there exists a ${\sigma}\in C^{\infty}_{-2}(\CC^2,\Theta)$ such that $\bar\partial {\sigma} = \bar\partial{{\zeta}}$. Average ${\sigma}$ by the group action of $\Gamma$ such that ${\sigma}$ is $\Gamma$-invariant, then ${\sigma}(0) = 0$.  Note that $\sigma$ is holomorphic in a neighborhood of $0$.
By standard argument, we can lift up ${\sigma}$ to $\tilde{\sigma} \in \Gamma(X,\Theta)$, where $\tilde{\sigma} = O(|z|^{-2})$. Then we have $\phi = \bar\partial(\chi\eta+\tilde{\sigma})$. Since $\chi\eta+\tilde{\sigma} = O(r^{-2})$, $\phi = 0$.

Finally, we are going to prove that
\begin{align}
 \H_{-3}(X,\Lambda^{0,1}\otimes\Theta)
\simeq \H_{ess}(X,\Lambda^{0,1}\otimes\Theta),
\end{align}
It is clear that $\H_{ess}(X,\Lambda^{0,1}\otimes\Theta)\subset \H_{-3}(X,\Lambda^{0,1}\otimes\Theta)$. To show the isomorphism, we need to show that $\VV = \{0\}$. Let $Y\in \WW$, by definition, $Y\in \H_{1}(X,\Theta)$. Then $Y$ has the following expansion as $r\to\infty$:
\begin{align}
Y = \sum_{i,j} a_{i,j}z_i\frac{\partial}{\partial z_j} + \sum_{k} b_k \frac{\partial}{\partial z_k} +  O(r^{-3 + \epsilon})
\end{align}
where $a_{i,j},b_k$ are constants. The decaying rate of $O(r^{-3 + \epsilon})$ comes from the fact that $\square$ has no indicial roots between $-2$ and $0$, and 
there is no $\bar\partial$-closed kernel element corresponding to the root of $-2$.

If $\Gamma$ is nontrivial, then $b_k=0$. Denote 
\begin{align}
\pi: X\mapsto \CC^2/\Gamma
\end{align}
as the minimal resolution of $\CC^2/\Gamma$. Since $ Z = \sum_{i,j}a_{i,j}z_i\frac{\partial}{\partial z_j}$ is a holomorphic vector field on $\CC^2/\Gamma\setminus\{0\}$ and $\lim_{r\to 0} Z = 0$, by standard argument, $Z$ can be lifted up to a holomorphic vector field $\tilde{Z}$ on $X$. Then $Y-\tilde{Z} \in \H_{-3}(X,\Theta)$. By using integration by parts, $Y-\tilde{Z}$ is holomorphic, then $Y$ is holomorphic, $\bar\partial Y = 0$, and $V=0$. Then $\H_{-3}(X,\Lambda^{0,1}\otimes\Theta)\simeq \H_{ess}(X,\Lambda^{0,1}\otimes\Theta)$.

 Finally, if $\Gamma$ is trivial, then $X$ is biholomorphic to $\CC^2$, and the 
coordinate vector fields extend to global holomorphic vector
fields, so these do not give any non-trivial elements of $\VV$. 
\end{proof}

\subsection{Proof of Theorem \ref{t1.8}}
\begin{proof}[Proof of Theorem \ref{t1.8}]
In the following, $X_t$ will stand for the triple $(X,g_t,J_t)$.
Since the index of a strongly continuous family of Fredholm operators is constant, 
the index $P(t)$ of the operator $P_t$ defined in Lemma \ref{l2.5} is locally constant along a path of ALE K\"ahler metrics. Consequently, $P(t)$ is constant along the path $X_t$. By the same argument as in Lemma \ref{l2.5}, the cokernel of $P_t$ is trivial for all $t \in [0,1]$. Therefore $P_t$ are all surjective Fredholm operators, 
with $\dim(\ker P_t) = \dim(\ker P_0)=\dim(\ker P_1)$. 
Therefore if there exists a path that connects $X_1$ with a minimal resolution $X_0$, then $\dim(\H_{-3}(X_1,\Lambda^{0,1}\otimes\Theta))=\dim(\H_{-3}(X_0,\Lambda^{0,1}\otimes\Theta))$. 
Since $\mathfrak{G}(X_1) = \{ e\}$, 
\begin{align}
\dim_{\RR}(\H_{ess}(X_1,\Lambda^{0,1}\otimes\Theta)) = \dim_{\RR}(\H_{-3}(X_0,\Lambda^{0,1}\otimes\Theta))-\dim(\mathfrak{G}(X_0)) = j_\Gamma.
\end{align}
In the non-hyperk\"ahler case, Lemma \ref{non-hyperKahler} implies that 
\begin{align}
m(X_1) = j_{\Gamma} + b_2(X) = m_\Gamma,
\end{align}
and Theorem \ref{t2} implies that the local moduli space of scalar-flat K\"ahler ALE metrics near $g_1$ is a smooth manifold of dimension $m_\Gamma$. 
As remarked above, 
the hyperk\"ahler moduli space is constructed globally by Kronheimer, and $m_{\Gamma} 
= 3k- 3$.
 \end{proof}

\bibliography{Han_Viaclovsky}

\end{document}